\documentclass[
reqno]{amsart}
\usepackage{latexsym,amsmath,amssymb,amscd}
\usepackage[all]{xy}

\def\today{\ifcase \month \or
   January \or February \or March \or April \or
   May \or June \or July \or August \or
   September \or October \or November \or December \fi
   \space\number\day , \number\year}

\makeatletter
  \newcommand\@dotsep{4.5}
  \def\@tocline#1#2#3#4#5#6#7{\relax
     \ifnum #1>\c@tocdepth 
     \else
     \par \addpenalty\@secpenalty\addvspace{#2}%
     \begingroup \hyphenpenalty\@M
     \@ifempty{#4}{%
     \@tempdima\csname r@tocindent\number#1\endcsname\relax
        }{%
         \@tempdima#4\relax
           }%
      \parindent\z@ \leftskip#3\relax \advance\leftskip\@tempdima\relax
      \rightskip\@pnumwidth plus1em \parfillskip-\@pnumwidth
       #5\leavevmode\hskip-\@tempdima #6\relax
       \leaders\hbox{$\m@th
       \mkern \@dotsep mu\hbox{.}\mkern \@dotsep mu$}\hfill
       \hbox to\@pnumwidth{\@tocpagenum{#7}}\par
       \nobreak
        \endgroup
         \fi}
\makeatother

\begin{document}

\makeatletter
\@addtoreset{figure}{section}
\def\thefigure{\thesection.\@arabic\c@figure}
\def\fps@figure{h,t}
\@addtoreset{table}{bsection}

\def\thetable{\thesection.\@arabic\c@table}
\def\fps@table{h, t}
\@addtoreset{equation}{section}
\def\theequation{
\arabic{equation}}
\makeatother

\newcommand{\bfi}{\bfseries\itshape}

\newtheorem{theorem}{Theorem}
\newtheorem{acknowledgment}[theorem]{Acknowledgment}
\newtheorem{algorithm}[theorem]{Algorithm}
\newtheorem{axiom}[theorem]{Axiom}
\newtheorem{case}[theorem]{Case}
\newtheorem{claim}[theorem]{Claim}
\newtheorem{conclusion}[theorem]{Conclusion}
\newtheorem{condition}[theorem]{Condition}
\newtheorem{conjecture}[theorem]{Conjecture}
\newtheorem{construction}[theorem]{Construction}
\newtheorem{corollary}[theorem]{Corollary}
\newtheorem{criterion}[theorem]{Criterion}
\newtheorem{data}[theorem]{Data}
\newtheorem{definition}[theorem]{Definition}
\newtheorem{example}[theorem]{Example}
\newtheorem{lemma}[theorem]{Lemma}
\newtheorem{notation}[theorem]{Notation}
\newtheorem{problem}[theorem]{Problem}
\newtheorem{proposition}[theorem]{Proposition}
\newtheorem{question}[theorem]{Question}
\newtheorem{remark}[theorem]{Remark}
\newtheorem{setting}[theorem]{Setting}
\numberwithin{theorem}{section}
\numberwithin{equation}{section}

\newcommand{\todo}[1]{\vspace{5 mm}\par \noindent
\framebox{\begin{minipage}[c]{0.85 \textwidth}
\tt #1 \end{minipage}}\vspace{5 mm}\par}

\renewcommand{\1}{{\bf 1}}

\newcommand{\hotimes}{\widehat\otimes}

\newcommand{\Ad}{{\rm Ad}}
\newcommand{\Alt}{{\rm Alt}\,}
\newcommand{\Ci}{{\mathcal C}^\infty}
\newcommand{\comp}{\circ}
\newcommand{\wt}{\widetilde}

\newcommand{\C}{\mathbb C}
\newcommand{\D}{\mathbb D}
\newcommand{\Hb}{\text{\bf H}}
\newcommand{\N}{\mathbb N}
\newcommand{\R}{\mathbb R}
\newcommand{\T}{\mathbb T}
\newcommand{\Ub}{\mathbb U}

\newcommand{\ph}{\text{\bf P}}
\newcommand{\de}{{\rm d}}
\newcommand{\ev}{{\rm ev}}
\newcommand{\fimes}{\mathop{\times}\limits}
\newcommand{\id}{{\rm id}}
\newcommand{\ie}{{\rm i}}
\newcommand{\Cp}{\textbf {CPos}\,}
\newcommand{\End}{{\rm End}\,}
\newcommand{\Gr}{{\rm Gr}}
\newcommand{\GL}{{\rm GL}}
\newcommand{\Hilb}{{\bf Hilb}\,}
\newcommand{\Hom}{{\rm Hom}}
\renewcommand{\Im}{{\rm Im}\,}
\newcommand{\Ker}{{\rm Ker}\,}
\newcommand{\GLH}{\textbf{GrLHer}}
\newcommand{\HLH}{\textbf{HomogLHer}}
\newcommand{\LH}{\textbf{LHer}}
\newcommand{\Kern}{\textbf {Kern}}
\newcommand{\Lie}{\textbf{L}}
\newcommand{\lf}{{\rm l}}
\newcommand{\pr}{{\rm pr}}
\newcommand{\Ran}{{\rm Ran}\,}
\renewcommand{\Re}{{\rm Re}\,}

\newcommand{\RK}{{\mathcal P}{\mathcal K}^{-*}}
\newcommand{\spann}{{\rm span}}
\newcommand{\SLH}{\textbf {StLHer}}
\newcommand{\Rg}{\textbf {RepGLH}}
\newcommand{\Rep}{\textbf {sRep}}

\newcommand{\Tr}{{\rm Tr}\,}
\newcommand{\Tran}{\textbf{Trans}}

\newcommand{\G}{{\rm G}}
\newcommand{\U}{{\rm U}}
\newcommand{\Gl}{{\rm GL}}
\newcommand{\SL}{{\rm SL}}
\newcommand{\SU}{{\rm SU}}
\newcommand{\VB}{{\rm VB}}

\newcommand{\Bc}{{\mathcal B}}
\newcommand{\Cc}{{\mathcal C}}
\newcommand{\Dc}{{\mathcal D}}
\newcommand{\Ec}{{\mathcal E}}
\newcommand{\Fc}{{\mathcal F}}
\newcommand{\Gc}{{\mathcal G}}
\newcommand{\Hc}{{\mathcal H}}
\newcommand{\Kc}{{\mathcal K}}
\newcommand{\Nc}{{\mathcal N}}
\newcommand{\Oc}{{\mathcal O}}
\newcommand{\Pc}{{\mathcal P}}
\newcommand{\Qc}{{\mathcal Q}}
\newcommand{\Rc}{{\mathcal R}}
\newcommand{\Sc}{{\mathcal S}}
\newcommand{\Tc}{{\mathcal T}}
\newcommand{\Uc}{{\mathcal U}}
\newcommand{\Vc}{{\mathcal V}}
\newcommand{\Xc}{{\mathcal X}}
\newcommand{\Yc}{{\mathcal Y}}
\newcommand{\Zc}{{\mathcal Z}}
\newcommand{\Ag}{{\mathfrak A}}
\renewcommand{\gg}{{\mathfrak g}}
\newcommand{\hg}{{\mathfrak h}}
\newcommand{\mg}{{\mathfrak m}}
\newcommand{\nng}{{\mathfrak n}}
\newcommand{\pg}{{\mathfrak p}}
\newcommand{\ug}{{\mathfrak u}}
\newcommand{\Gg}{{\mathfrak g}}
\newcommand{\Sg}{{\mathfrak S}}
\newcommand{\Ug}{{\mathfrak u}}

\markboth{}{}

\makeatletter
\title[Linear connections on vector bundles]
{Linear connections for reproducing kernels on vector bundles}
\author{Daniel Belti\c t\u a and  Jos\'e E. Gal\'e}
\address{Institute of Mathematics ``Simion
Stoilow'' of the Romanian Academy, Research Unit 1, 
P.O. Box 1-764, Bucharest, Romania}
\email{beltita@gmail.com, Daniel.Beltita@imar.ro}
\address{Departamento de matem\'aticas and I.U.M.A.,
Universidad de Zaragoza, 50009 Zaragoza, Spain}
\email{gale@unizar.es}
\thanks{This research was partly  supported by Project MTM2010-16679, DGI-FEDER, of the MCYT, Spain. 
The first-named author has also been supported by a Grant of the Romanian National Authority for Scientific Research, CNCS-UEFISCDI, project number PN-II-ID-PCE-2011-3-0131. 
The second-named author has also been supported by Project E-64, D.G. Arag\'on, Spain.}
\dedicatory{18 July 2013}

\makeatother

\begin{abstract}
We construct a canonical correspondence from a wide class of reproducing kernels 
on infinite-dimensional Hermitian vector bundles to linear connections on these bundles. 
The linear connection in question is obtained through a pull-back operation involving the tautological universal bundle and the classifying morphism of the input kernel. 
The aforementioned correspondence turns out to be a canonical functor between categories of kernels and linear connections. 
A number of examples of linear connections including the ones associated to classical kernels, homogeneous reproducing kernels and kernels occurring in the dilation theory for completely positive maps are given, together with their covariant derivatives.
\end{abstract}

\maketitle

\tableofcontents

\section{Introduction}

The theory of reproducing kernels and of their applications to the study of Lie group representations 
has undergone an impressive development over the years; 
see for instance the excellent monograph \cite{Ne00} and the references therein. 
The questions addressed in the present paper belong to the apparently not yet explored 
differential geometric aspects of this theory. 
More specifically we show that, under mild assumptions, 
the smooth reproducing kernels on an infinite-dimensional Hermitian vector bundle 
give rise to linear connections on that bundle, and this correspondence sets up 
a functor between suitably defined categories of reproducing kernels and linear connections, respectively. 
This functor also turns out to be canonical in some sense  
(Theorem~\ref{canonical}). 
In the case of the tautological vector bundle over the Grassmann manifold associated to a complex Hilbert space, 
the universal connection corresponds to the so-called universal reproducing kernel 
that we pointed out in the earlier paper \cite{BG11}. 
We also discuss a number of specific examples including the classical Hardy and Bergman kernels 
and others on infinite-dimensional manifolds.

The circle of ideas approached here is motivated by the interest in understanding 
certain physical models (\cite{Od88}, \cite{Od92}) 
as well as the geometric realizations 
for certain representations of groups of invertible elements in $C^*$-algebras; 
see for instance \cite{BR07}, \cite{BG08}, \cite{BG09} or \cite{Ne12}. 
Such realizations, of Borel-Weil type, were constructed by using suitable reproducing kernels 
on homogeneous vector bundles and a rich panel of differential geometric features of operator algebras   
turned out on this occasion. 
The ideas were developed in a categorial framework in \cite{BG11}, 
where the geometric features of reproducing kernels have been reinforced in relation with 
the geometry of tautological vector bundles taken as universal objects. 
We were thus naturally led to investigating the differential geometric features of reproducing kernels. 
The geometric significance of such kernels had been also pointed out for instance in \cite{BH98}  
and more recently, in the very fine paper \cite{Hi08}.

Section~\ref{Sect3} is devoted to the reductive structures in the framework of Banach-Lie groups. 
We briefly discuss here the linear connections induced by the reductive structures and 
we then present some examples related to the $C^*$-algebras 
and which are important for producing geometric realizations for 
representations of certain Banach-Lie groups. 
The reproducing kernels on Banach vector bundles that we deal with in this paper 
are discussed in Section~\ref{Sect4}. 
Our main constructions of linear connections out of reproducing kernels 
are presented in Section~\ref{sectconnect}. 
We compute their covariant derivative in terms of the input reproducing kernel 
(Theorem~\ref{deriv_prop3}), 
and we also study their functorial properties (Theorem~\ref{canonical}). 
And finally, a panel of significant examples of reproducing kernels is discussed in Section~\ref{examples}. 
Namely, we look at the usual type of operator-valued reproducing kernels, 
including the classical reproducing kernels of Hardy or Bergmann type, 
the homogeneous reproducing kernels that we earlier used 
in the geometric representation theory of Banach-Lie groups, 
or the kernels that are implicit in the dilation theory of completely positive maps on $C^*$-algebras. 
In Appendix~\ref{Sect2} we provide some auxiliary properties of connections on Banach bundles  
and we emphasize the operation of pull-back, which plays a key role throughout this paper. 

Let us mention that in the present paper we only show how to define the linear connection induced by a reproducing kernel 
and study the very basic properties of the correspondence between these two types of objects. 
Topics like the deep significance of such connections for the complex structures or 
$C^*$-Hermitian structures on infinite-dimensional vector bundles, the analysis of the linear connections associated with reproducing kernels arising in representations of semisimple Lie groups in function spaces, or applications to Cowen-Douglas operators
will be treated in forthcoming papers.

\section{Reductive structures for Banach-Lie groups}\label{Sect3}

In this section we introduce the abstract notion and example which, as inherent in universal bundles, 
will enable us in Section~\ref{sectconnect} to define connections associated with reproducing kernels on general vector bundles.

\subsection{Linear connections induced by reductive structures}
We first make the definition of the reductive structures we are interested in.  
Several versions of this notion showed up in the literature of differential geometry in infinite dimensions; 
see for instance \cite{MR92}, \cite{ACS95}, \cite{CG99}, and the references therein. 
Nevertheless it seems to us that, 
maybe due to the fact that the existing literature was largely motivated by problems involving operator algebras, 
one considered mainly reductive structures on homogeneous spaces of groups of invertible elements 
in unital associative Banach algebras. 
See however \cite[Def. 4.1]{Ne02} for the related general notion of a normed symmetric Lie algebra. 
We next introduce the reductive structures on the natural level of generality, 
which does not require Banach algebras or $C^*$-algebras but rather Banach-Lie groups. 

\begin{definition}\label{redestruc}
\normalfont
A {\it reductive structure} is a triple $(G_A,G_B;E)$ where $G_A$ is a real Banach-Lie group with Lie algebra $\Gg_A$, $G_B$ is a Banach-Lie subgroup of $G_A$ with Lie algebra $\Gg_B$,  and  $E\colon\Gg_A\to\Gg_A$ is a continuous linear map with the following properties:
 $E\circ E=E$; $\Ran E=\Gg_B$; and for every $g\in G_B$ the diagram
$$
\begin{CD}
\Gg_A @>{\Ad_{G_A}(g)}>> \Gg_A \\
@V{E}VV @VV{E}V \\
\Gg_A @>{\Ad_{G_A}(g)}>> \Gg_A
\end{CD}
$$
is commutative.

A {\it morphism of reductive structures} from $(G_A,G_B;E)$ to $(\wt G_A,\wt G_B;\wt E)$ is a homomorphism of Banach-Lie groups
$\alpha\colon G_A\to\wt G_A$ such that $\alpha(G_B)\subseteq\wt G_B$ and the diagram
$$
\begin{CD}
\Gg_A @>{d\alpha}>> \wt\Gg_A \\
@V{E}VV @VV{\wt E}V \\
\Gg_A @>{d\alpha}>> \wt\Gg_A
\end{CD}
$$
is commutative. 
For instance, a family of automorphisms of any reductive structure 
$(G_A,G_B;E)$ is provided by 
$\alpha_g\colon x\mapsto gxg^{-1}$, $G_A\to G_A$ ($g\in G_B$). 
\end{definition}

We will see in Theorem~\ref{conred} below that 
if $\rho$ is a uniformly continuous representation from $G_B$ on a Hilbert space $\Hc_B$,  
then any reductive structure for $G_A$ and $G_B$ as above gives rise to a connection 
on the homogeneous vector bundle $\Pi:D=\G_A\times_{G_B}\Hc_B\to G_A/G_B$ induced by $\rho$. 
Recall that 
$\G_A\times_{G_B}\Hc_B$ is the cartesian product $\G_A\times\Hc_B$ quotient the equivalence relation
$$
(g,h)\sim(g',h')\hbox{ iff\, there\, exists } w\in G_B \hbox{ such\ that } g'=gw, h'=\rho(w^{-1})h,
$$ 
endowed with its canonical structure of Banach manifold; see \cite{KM97a}. 
In order to make the statement, we first note that, on account of Remark~\ref{conninduc},  
the tangent bundle $\tau_D\colon TD\to D$ can be described as the mapping
$$
\tau_D\colon
(G_A\ltimes_{\Ad_{G_A}}\Gg_A)\times_{(G_B\ltimes_{\Ad_{G_B}}\Gg_B)}
(\Hc_B\oplus\Hc_B)\to G_A\times_{G_B}\Hc_B
$$
given by $ [((g,X),(f,h))]\mapsto [(g,f)]$.

\begin{theorem}\label{conred}
Let $(G_A,G_B,E)$ be a reductive structure and 
$\rho\colon G_B\to\Bc(\Hc_B)$ be   
a uniformly continuous representation. 
Then the homogeneous vector bundle
$\Pi\colon D=G_A\times_{G_B}\Hc_B\to G_A/G_B$ has a linear connection 
$\Phi_E\colon TD\to TD$ given by 
$$
[((g,X),(f,h))]
\mapsto[((g,E(X)),(f,h))]
=[((g,0),(f,d\rho(E(X))f+h))].
$$
\end{theorem}

\begin{proof}
First we check that the equality in the image of $\Phi$ holds. Take an arbitrary element $(u,Y)\in G_B\ltimes_{\Ad_{G_B}}\Gg_B$. 
Since 
$$
d\rho(Y)=\rho(u^{-1})d\rho(\Ad_{G_B}(u)Y)\rho(u),
$$
it follows from the matrix expression of $T\rho(u,Y)$ in Remark \ref{conninduc}
that 
$$
T\rho(u,Y)^{-1} 
=\begin{pmatrix}\rho(u^{-1})& 0 \cr 
-\rho(u^{-1}) d\rho(\Ad_{G_B}(u)Y) &\rho(u^{-1})\end{pmatrix},
$$
for every $u\in G_B$ and $Y\in\Gg_B$. Hence, 
if $(g,X)\in G_A\ltimes_{\Ad_{G_A}}\Gg_A$ we have 
$$
(g,E(X))\cdot(u,Y)=(g\,u,\Ad_{G_B}(u^{-1})E(X)+Y)
$$
and
$$
T\rho(u,Y)^{-1}\cdot(f,h)
=(\rho(u^{-1})f,-\rho(u^{-1})d\rho(\Ad_{G_B}(u)Y)f+
\rho(u^{-1})h).
$$
Then the equality of equivalence classes in the definition of $\Phi_E$ follows just taking $u=\1_B$ and $Y=-E(X)$.

Analogously, it is not difficult to check that the mapping $\Phi_E$ is well defined. 
In effect, for $(g,X)$ and $(u,Y)$ as above we have
$$
\begin{aligned}
((g\,u,E(\Ad(u^{-1})X)+Y)&,(T\rho)(u,Y)^{-1}(f,h))\\
&=((g\,u,\Ad(u^{-1})E(X)+Y),(T\rho)(u,Y)^{-1}(f,h))\\
&=((g,E(X))\cdot (u,Y),(T\rho)(u,Y)^{-1}(f,h))\\
&\sim ((a,E(X)),(f,h)),
\end{aligned}
$$
where we have used in the first equality the commutativity of the diagram in Definition~\ref{redestruc}.

The fact that $\Phi$ is smooth follows since 
$(g,X,f,h)\mapsto(g,E(X),f,h)$ is a smooth map on 
$G_A\times \Gg_A\times\Hc_B\times\Hc_B$, and the corresponding quotient map
$$(G_A\times \Gg_A)\times(\Hc_B\times\Hc_B)\to 
(G_A\ltimes_{\Ad}\Gg_A)\times_{(G_B\ltimes_{\Ad}\Gg_B)}(\Hc_B\oplus\Hc_B)=TD$$ 
is a submersion (see e.g., \cite[Th.~37.12]{KM97a}).

Finally, the connection properties are readily checked.
\end{proof}

\begin{definition}\label{defconind}
\normalfont
The connection $\Phi_E$ constructed in Theorem~\ref{conred} will be called the 
{\it linear connection induced by the reductive structure} $(G_A,G_B;E)$.
\end{definition}

\begin{remark}\label{prinandvert}
\normalfont
In the definition of the connection $\Phi_E$, the expression 
$$
\Phi_E\left([((g,X),(f,h))]\right)=[((g,E(X)),(f,h))]
$$ 
reflects the fact that $\Phi_E$ is a linear connection on the 
$\Pi\colon G_A\times_{G_B}\Hc_B\to G_A/G_B$ induced by the principal connection $E$ on the principal bundle $G_A\to G_A/G_B$. Complementarily, the presentation 
$$
\Phi_E\left([((g,X),(f,h))]\right)=[((g,0),(f,d\rho(E(X))f+h))]
$$ 
emphasizes on the fact that the range of the connection $\Phi_E$ lies in the vertical subbundle of 
$T(G_A\times_{G_B}\Hc_B)$.
\end{remark}

We will now show that the reductive structures and pull-backs of connections are compatible, 
in the sense that connections induced by reductive structures are invariant under the pull-back action. 
Specifically, let $\alpha\colon G_A\to\wt G_A$ be a morphism of reductive structures from
$(G_A,G_B,E)$ to $(\wt G_A,\wt G_B,\wt E)$. 
Let $\wt\rho_B\colon \wt G_B\to\Bc(\Hc_B)$ a uniformly continuous representation and define  $\rho_B:=\wt\rho_B\circ\alpha|_{G_B}$. 
Thus we can construct the homogeneous vector bundles
$\Pi\colon D=G_A\times_{G_B}\Hc_B\to G_A/G_B$ and
$\wt\Pi\colon \wt D=\wt G_A\times_{\wt G_B}\Hc_B\to \wt G_A/\wt G_B$ carrying the linear connections $\Phi_E$ and $\wt\Phi_{\wt E}$, respectively, induced by the corresponding reductive structures (see Definition~\ref{defconind}).

Set $\Theta=(\delta,\zeta)$ where
$$
\zeta\colon gG_B\mapsto \alpha(g)\wt G_B,\ G_A/G_B\to \wt G_A/\wt G_B
$$
and
$$
\delta\colon[(g,f)]\mapsto [( \alpha(g),f)],\   G_A\times_{G_B}\Hc_B
\to\wt G_A\times_{\wt G_B}\Hc_B.
$$
It is readily seen that the pair $\Theta=(\delta,\zeta)$ is a morphism of the bundle $\Pi$ into $\wt\Pi$. 
Let $\Theta^*(\wt\Phi_{\wt E})$ denote the pull-back of the connection $\wt\Phi_{\wt E}$ through $\Theta$, 
in accordance with Definition \ref{pullcon}.

\begin{proposition}\label{redumorf} 
In the above setting, $\Theta^*(\wt\Phi_{\wt E})=\Phi_E$.
\end{proposition}

\begin{proof}
The tangent map $T\delta\colon TD\to T\wt D$ is given by
$$
T\delta\colon[((g,X),(f,h))]\mapsto[((\alpha(g),d\alpha(X)),(f,h))],\ TD\to T\wt D.
$$
Therefore, by Definition~\ref{conred} we get
\allowdisplaybreaks
\begin{align}
((T\delta)\circ\Phi_E)([((g,X),(f,h))])
&=T\delta\left([((g,E(X)),(f,h))]\right) \nonumber\\
&=[((\alpha(g),d\alpha(E(X))),(f,h))] \nonumber\\
&=
[((\alpha(g),\wt E(d\alpha(X))),(f,h))] \nonumber\\
&=\wt\Phi_{\wt E}\left([((\alpha(g),d\alpha(X)),(f,h))] \right) \nonumber\\
&=(\wt\Phi_{\wt E}\circ T\delta)\left([((g,X),(f,h))]\right), \nonumber
\end{align}
for every $g\in G_A$, $X\in\Gg_A$, $f,h\in\Hc_B$,
where the third equality follows since $\alpha$ is a morphism of reductive structures. Thus $\Theta^*(\wt\Phi_{\wt E})=\Phi_E$ by Definition~\ref{pullcon}.
\end{proof}

In the following remark we sketch a method for computing the covariant derivative 
for the linear connection induced by a reductive structure 
in the particular case when the construction of the homogeneous vector bundle 
involves the restriction 
of a representation of the larger group.  
This computation in the finite-dimensional situation can be found in~\cite{BR90}. 
In the special case when $G_A$ is the group of invertible elements of some associative Banach algebra, 
the Maurer-Cartan form introduced below was also constructed in \cite[subsect. 3.1]{MR92}. 

\begin{remark}\label{MC_rem}
\normalfont
Let $(G_A,G_B;E)$ be a reductive structure and denote $\mg=\Ker E$, so that $\Gg_A=\Gg_B\dotplus\mg$ 
and $\Ad_{G_A}(G_B)\mg=\mg$. 

\smallbreak

\noindent (1) 
There exists a natural isomorphism of vector bundles $T(G_A/G_B)\simeq G_A\times_{G_B}\mg$ 
(see for instance the proof of \cite[Cor. 5.5]{BG08}), 
where the latter homogeneous bundle is defined by using the adjoint action of $G_B$ on $\mg$. 
It follows that the function 
$$\beta\colon G_A\times_{G_B}\mg\to\Gg_A,\quad  \beta([(g,X)])=\Ad_{G_A}(g)X$$
can be thought of as a $\Gg_A$-valued differential 1-form $\beta\in\Omega^1(G_A/G_B,\Gg_A)$, 
to be called the \emph{Maurer-Cartan form} of the reductive structure under consideration. 
This vector-valued differential form essentially comes from an embedding of the tangent vector bundle 
$T(G_A/G_B)$ into the trivial vector bundle $(G_A/G_B)\times\Gg_A$ over $G_A/G_B$. 
Specifically, we have $T(G_A/G_B)\simeq G_A\times_{G_B}\mg\hookrightarrow G_A\times_{G_B}\Gg_A$ 
and also the $G_A$-equivariant trivialization of vector bundles 
$G_A\times_{G_B}\Gg_A\simeq(G_A/G_B)\times\Gg_A$, $[(g,X)]\mapsto (gG_B,\Ad_{G_A}(g)X)$. 

\smallbreak

\noindent (2) 
A similar trivialization can be set up  
for any homogeneous vector bundle whose construction involves the restriction 
of a representation of the larger group. 
More precisely, if $\rho\colon G_A\to\Bc(\textbf{E})$ is a uniformly continuous representation 
on some Banach space $\textbf{E}$, then we have the (inverse to each other) 
isomorphisms of vector bundles over $G_A/G_B$, 
$$
G_A\times_{G_B}\textbf{E}\simeq (G_A/G_B)\times\textbf{E}
$$
given by $[(g,v)]\mapsto(gG_B,\rho(g)v)$ and $(gG_B,v)\mapsto[(g,\rho(g)^{-1}v)]$, respectively.  
Since there exist one-to-one correspondences between the smooth $\textbf{E}$-valued functions or differential forms on $G_A/G_B$ 
and the sections or differential forms with values in the trivial vector bundle 
$(G_A/G_B)\times\textbf{E}\to G_A/G_B$, 
we can use the above isomorphisms of vector bundles in order to define the covariant derivative of $\textbf{E}$-valued functions 
\begin{equation}\label{MC_rem_eq1}
\nabla\colon\Ci(G_A/G_B,\textbf{E})\to\Omega^1(G_A/G_B,\textbf{E})
\end{equation}
as the covariant derivative induced by the reductive structure $(G_A,G_B;E)$ and the representation $\rho$.

\smallbreak

\noindent (3) 
For every $F\in\Ci(G_A/G_B,\textbf{E})$ we denote by $(\de\rho\circ\beta).F\in\Omega^1(G_A/G_B,\textbf{E})$ 
the differential form defined for every $z\in G_A/G_B$ and $X\in T_z(G_A/G_B)$ by 
$$
((\de\rho\circ\beta).F)(X)=(\de\rho(\beta_z(X)))((z)),
$$ 
where $\beta_z:=\beta\vert_{T_z(G_A/G_B)}\colon T_z(G_A/G_B)\to\Gg_A$ 
and $\de\rho\colon\Gg_A\to\Bc(\textbf{E})$ is the derived representation. 
The method of proof of \cite[Prop. 1.1]{BR90}, 
which extends directly to the present infinite-dimensional case 
(see also \cite[Th. 37.23(9) and Th. 37.30--31]{KM97a}), 
leads to the following conclusion.
If $\rho\colon G_A\to\Bc(\textbf{E})$ is a representation as above, 
then the covariant derivative \eqref{MC_rem_eq1} 
can be computed for every $F\in\Ci(G_A/G_B,\textbf{E})$ by the formula 
$\nabla F=\de F-(\de\rho\circ\beta).F$, 
that is, 
$$
(\nabla F)([(g,X)])=(\de F)([(g,X)])
-\rho(g)\de\rho(X)\rho(g)^{-1}F(gG_B)
$$
for all $g\in G_A$ and $X\in\mg$, 
which follows by also using the expression of $\beta$ given in  
item~(1) above and the fact that 
$d\rho(\rm{Ad}_{G_A}(g)X)=\rho(g)\de\rho(X)\rho(g)^{-1}$ for all $g\in G_A$, $X\in\Gg_A$.   
\end{remark}

\subsection{Some reductive structures related to $C^*$-algebras}

Next we give some key examples of reductive structures and morphisms between them.

\begin{example}[Lie group representations]\label{ex1}
\normalfont
Let $(G_A,G_B;E)$ be a reductive structure and let 
$\rho_A\colon G_A\to \Bc(\Hc_A)$ be a uniformly continuous unitary representation such that 
$\rho_{A}|_{G_B}$ has a non-trivial 
invariant closed subspace $\Hc_B\subseteq\Hc_A$. Denote
$\rho_B(g):=\rho_A(g)|_{\Hc_B}$ for every
$g\in G_B$ and define
\begin{itemize}
\item $\widetilde G_A=\U(\Hc_A)$, the unitary operators on $\Hc_A$; 

\item $\widetilde G_B=\U(\Hc_A)\cap\{p\}'$, the subgroup of $\U(\Hc_A)$ formed by the operators commuting with the orthogonal 
projection $p$ on $\Hc_B$ (that is, the operators that leave $\Hc_B$ invariant);

\item $\widetilde E\colon \widetilde\Gg_A=\ug(\Hc_A)\to\widetilde\Gg_B= \ug(\Hc_A)\cap\{p\}'$, where 
$\ug(\Hc_A)=\{X\in\Bc(\Hc_A): X^*=-X\}$ and 
for $X\in\Bc(\Hc_A)$  we take
$\widetilde E(X):=pXp+(\1-p)X(\1-p)$.
\end{itemize}
Then the mapping $\rho_A\colon G_A\to\widetilde G_A$ is a morphism of reductive structures from $(G_A,G_B;E)$ to
$(\widetilde G_A, \widetilde G_B; \widetilde E)$.
\end{example}

\begin{example}[conditional expectations on $C^*$-algebras]\label{ex2}
\normalfont

Let $A$ be a unital $C^*$-algebra with a unital $C^*$-subalgebra $B$ 
for which there exists a conditional expectation $E\colon A\to B$. 
This means that $E$ is a linear projection on $A$ with $\Ran E=B$ and norm one. 
By Tomiyama's theorem we have moreover that
$$
E(b_1ab_2)=b_1E(a)b_2 \text{ and }E(b^*)=E(b)^*\quad (a\in A; b_1, b_2\in B),
$$
and additionally $E(\1_A)=\1_B\left(=\1_A\right)$. 
Let $\G_\Lambda$ denote for $\Lambda\in\{A,B\}$ the Banach-Lie group of invertibles in $\Lambda$ 
endowed with its norm topology. 
Then the Lie algebra of $\G_\Lambda$ is $\Gg_\Lambda=\Lambda$, with the element $X$ of $\Gg_\Lambda$ 
obtained by derivation of the path $e^{tX}$ at $t=0$.   
Since in this $C^*$ case we have that
$Ad(g)a=gag^{-1}$ for every $g\in \G_A$ and $a\in A$,
the expectation $E$ satisfies the conditions of Definition \ref{redestruc}, 
so that $(\G_A,\G_B;E)$ is a reductive structure.

If for two triples $(A,B;E)$, $(\widetilde A,\widetilde B;\widetilde E)$ as above 
we also have a bounded $\ast$-homomorphism
$\phi\colon A\to \widetilde A$ satisfying $\phi\circ E= \widetilde E\circ\phi$ 
then $\alpha:=\phi|_{\G_A}$ defines a morphism 
between the reductive structures $(\G_A,\G_B;E)$ and 
$(\G_{\widetilde A},\G_{\widetilde B};\widetilde E)$. 

The reductive structures $(\U_A,\U_B;E\vert_{\ug_A})$ and 
$(\U_{\widetilde A},\U_{\widetilde B};\widetilde E\vert_{\ug_{\widetilde A}})$ 
defined by the unitary groups have a similar property, 
where we recall that the unitary group $\U_A=\{u\in A\mid uu^*=u^*u=\1\}$ 
is a Banach-Lie group whose Lie algebra is $\ug_A=\{x\in A\mid x^*=-x\}$, 
and similarly for the other unitary groups involved here. 
\end{example}

\begin{example}[completely positive maps]\label{ex3}
\normalfont
Let $A$ be a unital $C^*$-algebra and $\Hc_0$ be a complex Hilbert space. 
Recall that a unital linear mapping $\Psi\colon A\to\Bc(\Hc_0)$ is said to be completely positive if 
the linear mapping 
$$
\Psi_n:=\Psi\otimes\id_{M_n(\C)}\colon M_n(A)=A\otimes M_n(\C)\to M_n(\Bc(\Hc_0))
$$
is positive (i.e., it maps positive elements to positive elements) for all $n\ge1$. 
It is well known that every conditional expectation is a completely positive map.

By the Stinespring dilation procedure, 
for a given completely positive map $\Psi\colon A\to\Bc(\Hc_0)$, there are a Hilbert space $\Hc$, an isometry $V\colon \Hc\to\Hc_0$ and a unital $*$-representation of $C^*$-algebras $\lambda\colon A\to\Bc(\Hc)$, which is induced by the left multiplication in $A$, such that
$$
\Psi(a)=V^*\lambda(a)V\qquad (a\in A).
$$
Then the representation $\lambda$ is called a Stinespring dilation or representation associated with $\Psi$. 
Details can be found for instance in \cite{Pa02}.

Assume that there is a conditional expectation $E\colon A\to A$, with $B:=E(A)$, such that $\Psi\circ E=\Psi$. 
In \cite {BG08}, it was noted that if $\sigma_A\colon A\to \Bc(\Hc_A)$ and $\sigma_B\colon B\to \Bc(\Hc_B)$ 
are the minimal Stinespring representations associated with
$\Psi$ and $\Psi|_B$, respectively, as above then $\Hc_B\subseteq\Hc_A$, 
and the orthogonal projection $P\colon\Hc_A\to\Hc_B$ is induced by $E\colon A\to B$ in the Stinespring construction. 
Thus, for the given representation $\rho\colon \G\to\Bc(\Hc_B)$,  
we get the connection on the homogeneous bundle
$\G_A\times_{\G_B}\Hc_B\to \G_A/\G_B$ yielded by $E$ as in Definition~\ref{defconind}.
\end{example}

\begin{example}[universal bundles]\label{ex4}
\normalfont
Let $\Hc$ be a complex Hilbert space. 
The {\it Grassmann manifold} of $\Hc$ is 
$$\Gr(\Hc):=\{\Sc\mid\Sc\text{ closed linear subspace of }\Hc\}.$$ 
It is well known that it has the structure of a complex Banach manifold 
(see also \cite{DG01}). 
The set 
$\Tc(\Hc):=\{(\Sc,x)\in\Gr(\Hc)\times\Hc\mid x\in\Sc\}\subseteq\Gr(\Hc)\times\Hc$ 
is also a complex Banach manifold, and the mapping  
$\Pi_{\Hc}\colon\,(\Sc,x)\mapsto\Sc$, $\Tc(\Hc)\to\Gr(\Hc)$ is 
a holomorphic Hermitian vector bundle on which $U(\Hc)$ acts  
holomorphically (non-transitively on the base $\Gr(\Hc)$ if $\dim\Hc\ge2$);  
see \cite[Ex. 3.11 and 6.20]{Up85}. 
We call $\Pi_{\Hc}$ the {\it universal (tautological) vector bundle} associated with  the Hilbert space~$\Hc$. 
A canonical connection can be defined on that bundle, which relies on the preceding examples. 
To see this, let us consider the connected components of $\Gr(\Hc)$.  

For every $\Sc\in\Gr(\Hc)$ denote by $p_{\Sc}\colon\Hc\to\Sc$ the corresponding orthogonal projection. 
Take $\Sc_0\in\Gr(\Hc)$ and put $p:=p_{\Sc_0}$.  
The connected component of $\Sc_0\in\Gr(\Hc)$ is given by
$$
\begin{aligned}
\Gr_{\Sc_0}(\Hc)
&=
\{u\Sc_0\mid u\in U(\Hc)\} 
=\{\Sc\in\Gr(\Hc)\mid\dim\Sc=\dim\Sc_0,\, 
\dim\Sc^\perp=\dim\Sc_0^\perp\} \\
&\simeq U(\Hc)/\Uc(p)\simeq U(\Hc)/(U(\Sc_0)\times U(\Sc_0^\perp))
\end{aligned}
$$
where $\Uc(p):=\{u\in U(\Hc)\mid u\Sc_0=\Sc_0\}$.   
(See for instance \cite[Prop. 23.1]{Up85} or \cite[Lemma~4.3]{BG09}.)

By restricting  $\Pi_{\Hc}$ to $\Tc_{\Sc_0}(\Hc):=\{(\Sc,x)\in\Tc(\Hc)\mid\Sc\in\Gr_{\Sc_0}(\Hc)\}$ 
we obtain  
the 
Hermitian bundle $\Pi_{\Hc,\Sc_0}\colon\Tc_{\Sc_0}(\Hc)\to\Gr_{\Sc_0}(\Hc)$.  
The map  
$$
U(\Hc)\times_{\Uc(p)}\Sc_0\ni[(u,x)]\mapsto(u\Sc_0,ux)\in\Tc_{\Sc_0}(\Hc)
$$ 
is a diffeomorphism of vector bundles between 
$U(\Hc)\times_{\Uc(p)}\Sc_0\to\Uc/\Uc(p)$, 
where the representation of $\Uc(p)$ on $\Sc_0$ is just the tautological action,
and   
$\Pi_{\Hc,\Sc_0}\colon\Tc_{\Sc_0}(\Hc)\to\Gr_{\Sc_0}(\Hc)$. 
See \cite[Prop. 4.5]{BG09}. 

On the other hand, we have a conditional expectation 
$$
E_p\colon \Bc(\Hc)\to\{p\}',\quad X\mapsto E_p(X):=pXp+(1-p)X(1-p),
$$
where $\{p\}'=\{X\in\Bc(\Hc)\mid Xp=pX\}$. 
By restricting $E_p$ to the Lie algebra $\ug(\Hc):=\{X\in\Bc(\Hc)\mid X^*=-X\}$ of $U(\Hc)$  
and applying Theorem \ref{conred} to the reductive structure
$(U(\Hc),\Uc(p);E_p\mid_{\ug(\Hc)})$, we obtain
the following canonical connection on the universal bundle 
$\Tc(\Hc)\to\Gr(\Hc)$. 
Recall that $\Gr_{\Sc}(\Hc)$, for $\Sc$ running over $\Gr(\Hc)$, are the connected components of $\Gr(\Hc)$. 
\end{example}

\begin{definition}\label{tautoconn} 
\normalfont
In the framework of Example~\ref{ex4}, the {\it universal (linear) connection} 
on the tautological bundle $\Pi_{\Hc,\Sc_0}\colon\Tc_{\Sc_0}(\Hc)\to\Gr_{\Sc_0}(\Hc)$ 
is the mapping
$$
\Phi_{\Sc_0}\colon T(\Tc_{\Sc_0}(\Hc))\to T(\Tc_{\Sc_0}(\Hc))
$$ 
given by 
$$
[((u,X), (x,y))]\mapsto 
[((u,E_p(X)), (x,y))]
=
[((u,0), (x,E_p(X)x+y))],
$$
for $u\in U(\Hc)$, $X\in\ug(\Hc)$, and $x, y\in\Sc_0$. 
Then the  {\it universal connection} $\Phi_{\Hc}$ 
on the tautological bundle $\Pi_{\Hc}\colon\Tc(\Hc)\to\Gr(\Hc)$ is defined by
$$ 
\Phi_{\Hc}((u,X)):=\Phi_{\Sc_0}((u,X))
$$
for every $(u,X)\in T(\Tc(\Hc))$ with $(u,X)\in T(\Tc_{\Sc_0}(\Hc))$, 
$\Sc_0\in\Gr(\Hc)$.
\end{definition}

\begin{remark}\label{unitaucon} 
\normalfont
The expression of the connection $\Phi_{\Hc}$ on the sub-bundle $\Pi_{\Hc,\Sc_0}$ 
depends obviously on the realization of $\Pi_{\Hc,\Sc_0}$ as a homogeneous vector bundle, 
given by the diffeomorphism 
$\Tc_{\Sc_0}(\Hc)=\Tc_{\Sc}(\Hc)\equiv U(\Hc)\times_{U(p_{\Sc})}\Sc$, 
for $\Sc$ running over the component $\Gr_{\Sc_0}(\Hc)$. 
However, we can say that $\Phi_{\Sc_0}$ is unique in the sense that it is invariant under the action of $U(\Hc)$:

Let $\alpha\in U(\Hc)$ so that $\Sc_1:=\alpha\Sc_0\in\Gr_{\Sc_0}(\Hc)$. 
Then there is the natural diffeomorphism 
$$ 
U(\Hc)\times_{U(p_{\Sc_0})}\Sc_0\rightarrow U(\Hc)\times_{U(p_{\Sc_1})}\Sc_1,\quad 
[(u,x_0)]\mapsto[(\alpha u\alpha^{-1},\alpha x_0)],
$$
which induces the diffeomorphism 
$T_\alpha\colon T(\Tc_{\Sc_0}(\Hc))\equiv T(\Tc_{\Sc_1}(\Hc))$, 
between the corresponding homogeneous tangent bundles, given by 
$$
T_\alpha\colon[((u,X),(x_0,y_0))]
\mapsto[((u\alpha^{-1},\alpha X\alpha^{-1}),(\alpha x_0,\alpha y_0))] 
$$
Then it is readily seen that, on the level of connections, 
$\Phi_{\Sc_1}=T_\alpha\circ\Phi_{\Sc_0}\circ (T_\alpha)^{-1}$, which is to say,  
$\Phi_{\Hc}$ is $U(\Hc)$-equivariant.

In the sequel, and particularly in what concerns Theorem \ref{canonical}, 
whenever we deal with connections on $T(\Tc(\Hc))$ 
we will be assuming that an element $\Sc_0$ has been fixed in every connected component 
$\Gr_{\Sc_0}(\Hc)$ of $\Gr(\Hc)$, 
and that the connection $\Phi_{\Sc_0}$ is referred to that element as indicated above. 
\end{remark}

Universal connections on finite-dimensional bundles were studied in several papers including 
for instance \cite{NR61}, \cite{NR63}, \cite{Sch80}, and \cite{PR86}. 

It turns out that the Grassmannian objects that we are considering here fit also well 
in the setting given in the above Example~\ref{ex3}. 
In fact, the compression mapping
$$  
\Psi_{\Sc_0}\colon X\mapsto p\circ X\circ\iota_{\Sc_0},\ \Bc(\Hc)\to\Bc(\Sc_0),
$$
where $\iota_{\Sc_0}$ is the inclusion $\Sc_0\hookrightarrow\Bc(\Hc)$, 
is a unital completely positive mapping satisfying 
$\Psi_{\Sc_0}\circ E_p=\Psi_{\Sc_0}$ and the vector bundle defined by
$(\Bc(\Hc), \{p\}',E_p;\Psi_{\Sc_0})$ as in Example~\ref{ex3} coincides with $\Pi_{\Hc,\Sc_0}$.

\smallbreak
We now provide a formula for the covariant derivative 
corresponding to the universal connection on a tautological bundle, 
which will be needed in the proof of a more general result of this type, 
given in Theorem~\ref{deriv_prop3} below. 
Related finite-dimensional formulas are implicit in \cite[Ch. III]{Le68} and \cite{PR86}, 
but we should mention that the approach to the following result is quite different in a couple of respects, 
beyond the obvious fact that we are working here in infinite dimensions. 
More specifically, our starting point is a connection defined as a splitting 
of the tangent space of the tautological vector bundle, rather than the corresponding covariant derivative 
as in the aforementioned references.  
Secondly, the following statement and proof emphasize the role of the orthogonal projections on closed subspaces 
in order to compute the covariant derivative. 
As the orthogonal projections are just the basic pieces of the  universal reproducing kernels 
(see Example~\ref{univker} below), we thus have an illustration of the main theme of the present paper, 
namely that the reproducing kernels give rise to linear connections of the bundles where these kernels live. 
That is nontrivial even in the case of the tautological bundles associated with finite-dimensional Hilbert spaces, 
and yet we were unable to find any reference for that relationship in the earlier literature. 

\begin{proposition}\label{univ_deriv}
Let $\Sc_0\in\Gr(\Hc)$. 
If $\sigma\in\Omega^0(\Gr_{\Sc_0}(\Hc),\Tc_{\Sc_0}(\Hc))$ is a smooth section, 
then there exists a unique smooth function $F_{\sigma}\in\Ci(\Gr_{\Sc_0}(\Hc),\Hc)$ 
such that $\sigma(\cdot)=(\,\cdot\,,F_{\sigma}(\cdot))$ and we have 
$$
\nabla\sigma(X)=(\Sc,p_{\Sc}(\de F_{\sigma}(X))), \quad \Sc\in\Gr_{\Sc_0}(\Hc), X\in T_{\Sc}(\Gr_{\Sc_0}(\Hc)),
$$ 
where $p_{\Sc}$ is the orthogonal projection from $\Hc$ onto~$\Sc$. 
\end{proposition}

\begin{proof}
We use the tautological representations 
$$
\U(\Sc_0)\times\U(\Sc_0^{\perp})
\mathop{\hookrightarrow}^{\rho_0}\Bc(\Sc_0)
\text{ and }
\U(\Sc_0)\times\U(\Sc_0^{\perp})\hookrightarrow\U(\Hc)
\mathop{\hookrightarrow}^{\rho_1}\Bc(\Hc)
$$
for constructing the homogeneous vector bundles 
$$
\begin{aligned}
\Pi_0 & \colon D_0:=\U(\Hc)\times_{\U(\Sc_0)\times\U(\Sc_0^{\perp})}\Sc_0\to\U(\Hc)/(\U(\Sc_0)\times\U(\Sc_0^{\perp})) 
\text{ and }\\
\Pi_1 & \colon D_1:=\U(\Hc)\times_{\U(\Sc_0)\times\U(\Sc_0^{\perp})}\Hc\to\U(\Hc)/(\U(\Sc_0)\times\U(\Sc_0^{\perp})).
\end{aligned}
$$ 
Then Remark~\ref{MC_rem}(2) provides a $\U(\Hc)$-equivariant diffeomorphism 
$$\delta_{\Sc_0}\colon\U(\Hc)\times_{\U(\Sc_0)\times\U(\Sc_0^{\perp})}\Hc\to 
(\U(\Hc)/(\U(\Sc_0)\times\U(\Sc_0^{\perp})))\times\Hc$$
which together with the natural diffeomorphism 
$$\zeta_{\Sc_0}\colon \U(\Hc)/({\U(\Sc_0)\times\U(\Sc_0^{\perp}))}\to\Gr_{\Sc_0}(\Hc)$$
provide an isomorphism between $\Pi_1$ and the trivial bundle $\Gr_{\Sc_0}(\Hc)\times\Hc\to\Gr_{\Sc_0}(\Hc)$. 
Also, $\Pi_0$ is a $\U(\Hc)$-homogeneous subbundle of $\Pi_1$ and 
the pair $(\delta_{\Sc_0},\zeta_{\Sc_0})$ restricts to an isomorphism 
from $\Pi_0$ onto the tautological bundle $\Pi_{\Hc,\Sc_0}\colon\Tc_{\Sc_0}(\Hc)\to\Gr_{\Sc_0}(\Hc)$. 

Now for $j=0,1$ let $\Phi_j\colon T(D_j)\to T(D_j)$ denote the linear connection 
induced by the reductive structure $(U(\Hc),\U(\Sc_0)\times\U(\Sc_0^{\perp});E_{p_{\Sc_0}})$ 
and $\nabla_j$ 
be the corresponding covariant derivative. 
 It easily follows by Theorem~\ref{conred} that $\Phi_1\vert_{T(D_0)}=\Phi_0$, 
 and then Proposition~\ref{deriv_prop1} shows that 
 $\nabla_1$ agrees with $\nabla_0$. 
 By taking into account the aforementioned isomorphisms of homogeneous vector bundles, 
 it then follows that the covariant derivative $\nabla$ in the tautological vector bundle $\Pi_{\Sc_0,\Hc}$ 
 agrees with the covariant derivative $\widetilde{\nabla}$ in the larger trivial bundle $\Gr_{\Sc_0}(\Hc)\times\Hc\to\Gr_{\Sc_0}(\Hc)$, 
 both these covariant derivatives being the ones induced by the reductive structure $(U(\Hc),\U(\Sc_0)\times\U(\Sc_0^{\perp});E_{p_{\Sc_0}})$. 
Consequently it suffices to compute the action of $\widetilde{\nabla}$ on the sections 
of the subbundle $\Pi_{\Sc_0,\Hc}$, and to this end we use Remark~\ref{MC_rem}(3).

If we write the operators on $\Hc$ as $2\times 2$ block matrices corresponding to the orthogonal decomposition $\Hc=\Sc_0\oplus\Sc_0^\perp$, then 
$$
\mg=\Ug(\Hc)\cap\Ker E_{p_{\Sc_0}}
=\Bigl\{\begin{pmatrix} \hfil 0 & \hfil R \\ \hfil -R^* & \hfil 0\end{pmatrix}\mid 
R\in\Bc(\Sc_0,\Sc_0^\perp)\Bigr\}
$$
hence for every $V\in\mg$ we have $V\Sc_0\subseteq\Sc_0^\perp$. 
Therefore, if we denote by 
$$\beta\colon T(\Gr_{\Sc_0}(\Hc))\simeq T(\U(\Hc)/(\U(\Sc_0)\times\U(\Sc_0^{\perp})))\to\Ug(\Hc)$$
the Maurer-Cartan form for the reductive structure 
$(U(\Hc),\U(\Sc_0)\times\U(\Sc_0^{\perp});E_{p_{\Sc_0}})$, as in Remark~\ref{MC_rem}(3), 
then for every $X\in T_{\Sc_0}(\Gr_{\Sc_0}(\Hc))$ we have 
$(\de\rho_1\circ\beta)(X)\Sc_0\subseteq\Sc_0^\perp$. 
On the other hand for 
$\sigma(\cdot)=(\cdot,F_{\sigma}(\cdot))
\in\Omega^0(\Gr_{\Sc_0}(\Hc),\Tc_{\Sc_0}(\Hc))$ 
as in the statement we have 
$(\widetilde{\nabla}\sigma)(X)
=(\Sc_0,(\widetilde{\nabla}F_{\sigma})(X))\in\{\Sc_0\}\times\Sc_0$, 
hence the equality provided by Remark~\ref{MC_rem}(3)
$$
\de F_{\sigma}(X)
=(\widetilde{\nabla}F_{\sigma})(X)+(\de\rho_1\circ\beta)(X)F_{\sigma}(\Sc_0)
$$
actually gives the decomposition corresponding to the orthogonal direct sum $\Hc=\Sc_0\oplus\Sc_0^\perp$. 
Therefore $(\widetilde{\nabla}F_{\sigma})(X)=p_{\Sc_0}(\de F_{\sigma}(X))$, 
and this proves the assertion for $\Sc=\Sc_0$ since we have seen above that 
$\nabla$ agrees with $\widetilde{\nabla}$. 

The formula for the covariant derivative $\nabla$ at another point $\Sc\in\Gr_{\Sc_0}(\Hc)$ then follows 
by using the transitive action of $\U(\Hc)$ on $\Gr_{\Sc_0}(\Hc)$ 
and the $\U(\Hc)$-equivariance property of the Maurer-Cartan form~$\beta$, 
and this completes the proof. 
\end{proof}

The above Examples \ref{ex3}--\ref{ex4} will be revisited in Section \ref{examples}.

\section{Reproducing kernels and their classifying morphisms}\label{Sect4}

In this section we begin the developments that will lead up in the next section 
to the canonical correspondence between the admissible reproducing kernels and the linear connections. 
More specifically, we will establish the basic properties of the classifying morphisms, 
which are bundle morphisms into the universal bundles over 
the Grassmann manifolds of the reproducing kernel Hibert spaces. 

\subsection{Reproducing kernels on Hermitian bundles}

Geometric models for representations of unitary groups of $C^*$-algebras were obtained in \cite{BR07} 
by using reproducing kernels associated with suitable homogeneous vector bundles.  
An approach to these topics in the framework of category theory was carried out in \cite{BG11}, 
which enables us to recover reproducing kernels on Hermitian vector bundles 
from the universal reproducing kernels on the tautological vector bundles; 
see Example~\ref{univker} and Theorem~\ref{univinvol} below.

\begin{definition}\label{like}
\normalfont
Let $Z$ be a Banach manifold.
A {\it Hermitian structure} on a smooth Banach vector bundle $\Pi\colon D\to Z$ is a family
$\{(\cdot\mid\cdot)_z\}_{z\in Z}$
with the following properties:
\begin{itemize}
\item[{\rm(a)}]
For every $z\in Z$,
$(\cdot\mid\cdot)_z\colon D_z\times D_z\to{\mathbb C}$
is a scalar product (${\mathbb C}$-linear in the first variable) 
that turns the fiber $D_z$ into a complex Hilbert space. 
\item[{\rm(b)}]
If $V$ is any open subset of $Z$,
and 
$\Psi_V\colon V\times\Ec\to\Pi^{-1}(V)$
is a trivializations (whose typical fiber is the complex Hilbert space $\Ec$) 
of the vector bundle $\Pi$
over $V$,  
then the function
$(z,x,y)\mapsto(\Psi_V(z,x)\mid\Psi_V(z,y))_z$,
$V\times \Ec\times\Ec\to{\mathbb C}$
is smooth.
\end{itemize}
A {\it Hermitian bundle} is a bundle endowed with a Hermitian structure as above.
\end{definition}

\begin{definition}\label{reprokernel}
\normalfont
Let $\Pi\colon D\to Z$ be a Hermitian bundle. 
A {\it reproducing kernel} on~$\Pi$ is a continuous section of the bundle 
$\rm{Hom}(p_2^*\Pi,p_1^*\Pi)\to Z\times Z$ such that the
mappings 
$K(s,t)\colon D_t\to D_s$ ($s,t\in Z$)
are bounded linear operators 
and such that $K$ is positive definite in the following sense: 
For every $n\ge1$ and $t_j\in Z$,
$\eta_j\in D_{t_j}$ ($j= 1,\dots, n$),
\begin{equation}\label{reprokernel_eq1}
\sum_{j,l=1}^n
\bigl(K(t_l,t_j)\eta_j\mid \eta_l\bigr)_{t_l}\ge0.
\end{equation}
Here $p_1,p_2\colon Z\times Z\to Z$ are the natural projection mappings. 

For every $\xi\in D$ we set $K_\xi:=K(\cdot,\Pi(\xi))\xi\colon Z\to D$, 
which is a section of the bundle $\Pi$.  
For $\xi,\eta\in D$, 
the prescriptions 
\begin{equation}\label{reprokernel_eq2}
(K_\xi\mid K_\eta)_{\Hc^K}:=(K(\Pi(\eta),\Pi(\xi))\xi\mid\eta)_{\Pi(\eta)},
\end{equation}
define an inner product 
$(\cdot\mid \cdot)_{\Hc^K}$ on $\hbox{span}\{K_\xi:\xi\in D\}$ 
whose completion gives rise to a Hilbert space denoted by $\Hc^K$, 
which consists of sections of the bundle $\Pi$ 
(see \cite{Ne00} or \cite[Th. 4.2]{BR07}). 
We also define the mappings 
$$\begin{aligned}
\widehat{K}& \colon D\to\Hc^K,\quad \widehat{K}(\xi)=K_\xi,\\
\zeta_K & \colon Z\to\Gr(\Hc^K),\quad \zeta_K(s)=\overline{\widehat{K}(D_s)},
\end{aligned}
$$
where the bar over $\widehat{K}(D_s)$ indicates the topological closure.
\end{definition}

In the following two lemmas we establish some basic properties of the above mappings. 

\begin{lemma}\label{smoothhat}
In the setting of Definition~\ref{reprokernel},  
if $K$ is a smooth section of the bundle $\rm{Hom}(p_2^*\Pi,p_1^*\Pi)$, then 
the mapping $\widehat{K}\colon D\to\Hc^K$ is smooth. 
\end{lemma}

\begin{proof}
Since both $K\colon Z\times Z\to\rm{Hom}(p_2^*\Pi,p_1^*\Pi)$ and $\Pi\colon D\to Z$ 
are smooth mappings, it follows by \eqref{reprokernel_eq2} that the function 
$$D\times D\to\C,\quad (\xi,\eta)\mapsto (\widehat{K}(\xi)\mid \widehat{K}(\eta))_{\Hc^K}$$
is smooth. 
Then the assertion follows by \cite[Th. 7.1]{Ne10}. 
\end{proof}

\begin{lemma}\label{equivKinverse}
In the setting of Definition~\ref{reprokernel}, the following assertions are equivalent at each $s\in Z$:
\begin{itemize}
\item[{\rm(i)}]
The operator $\widehat{K}\vert_{D_s}\colon D_s\to\Hc^K$ is injective and has closed range.
\item[{\rm(ii)}] The operator $K(s,s)\in\Bc(D_s)$ is invertible.
\end{itemize}
\end{lemma}

\begin{proof} 
The property that $\widehat{K}\vert_{D_s}\colon D_s\to\Hc^K$ is injective and has closed range 
is equivalent to the fact that  
there exists $c>0$ such that 
for every $\xi\in D_s$ we have $\Vert\widehat{K}(\xi)\Vert_{\Hc^K}\ge c\Vert \xi\Vert_{D_s}$, 
which is further equivalent to 
$(K_\xi\mid K_\xi)_{\Hc^K}\ge c^2 \Vert \xi\Vert_{D_s}^2$, 
that is, $(K(s,s)\xi\mid\xi)_{D_s}\ge c^2 \Vert \xi\Vert_{D_s}^2$. 
The latter condition is equivalent to the fact that $K(s,s)$ is invertible on $D_s$, 
since $K(s,s)$ is always a bounded nonnegative self-adjoint operator on the complex Hilbert space $D_s$, 
as a consequence of \eqref{reprokernel_eq1} in Definition~\ref{reprokernel} for $n=1$. 
\end{proof}

\begin{definition}\label{adm_def}
\normalfont
A reproducing kernel $K$ on the Hermitian bundle $\Pi\colon D\to Z$ 
 is called \emph{admissible} 
if it has the following properties: 
\begin{itemize}
\item[(a)] The kernel $K$ is smooth as a section of the bundle $\rm{Hom}(p_2^*\Pi,p_1^*\Pi)$. 
\item[(b)] For every $s\in Z$ the operator $K(s,s)\in\Bc(D_s)$ is invertible. 
\item[(c)] The mapping $\zeta_K\colon Z\to\Gr(\Hc^K)$ is smooth.
\end{itemize} 
\end{definition}


\begin{example}
\normalfont
Assume $\Pi\colon D\to Z$ is a Hermitian bundle whose fibers are finite dimensional 
(for instance, $\Pi$ is a line bundle). 
If a reproducing kernel $K$ on $\Pi$ satisfies the conditions (a)--(b) in Definition~\ref{adm_def}, 
then it also satisfies the condition~(c) hence it is an admissible reproducing kernel.  

To prove this, let $s_0\in Z$ arbitrary. 
Since the bundle $\Pi$ is locally trivial and its fibers are finite-dimensional, 
it follows by an application of Lemma~\ref{equivKinverse} 
that there exist a positive integer $n\ge 1$ and an open neighborhood $Z_0$ of $s_0\in Z$  
such that $\dim\zeta_K(s)=\dim D_s=n$ for every $s\in Z_0$. 
Then we can use \cite[Th. 5.5]{BG11} 
to obtain that the mapping $\zeta_K\colon Z\to\Gr(\Hc^K)$ is continuous. 

Next, for arbitrary $s_0\in Z$, $K(s_0,s_0)\in \Bc(D_{s_0})$ is an invertible operator.  
By considering a local trivialization of $\Pi$ near $s_0$ with the typical fiber $D_{s_0}$ 
and using the fact that $K\colon Z\times Z\to\rm{Hom}(p_2^*\Pi,p_1^*\Pi)$ is continuous, 
it follows that there exists an open neighborhood $Z_0$ of $s_0\in Z$ 
such that for arbitrary $s,t\in Z_0$ the operator $K(s,t)\in\Bc(D_t,D_s)$ is invertible. 
Let us define 
$\widetilde{K}_0\colon Z_0\to\Bc(D_{s_0},\Hc^K)$, $\widetilde{K}_0(s)=\widehat{K}\circ K(s,s_0)$. 
Since $K\colon Z\times Z\to\rm{Hom}(p_2^*\Pi,p_1^*\Pi)$ is smooth by hypothesis, 
$\dim D_{s_0}<\infty$, 
and $\widehat{K}\colon D\to\Hc^K$ is smooth by Lemma~\ref{smoothhat}, 
it follows that the mapping $\widetilde{K}_0\colon Z_0\to\Bc(D_{s_0},\Hc^K)$ is smooth. 

Moreover, for arbitrary $s\in Z_0$, the operator $K(s,s_0)\in\Bc(D_{s_0},D_s)$ 
is invertible, hence $\Ran(\widetilde{K}_0(s))=\widehat{K}(D_s)=\zeta_K(s)$. 
Consequently we have a smooth mapping $\widetilde{K}_0\colon Z_0\to\Bc(D_{s_0},\Hc^K)$ 
with the property that $\Ran(\widetilde{K}_0(\cdot))\colon Z_0\to\Gr(\Hc^K)$ is continuous. 
It then follows (see for instance \cite[Subsect. 1.8 and 1.5]{MS97}) that 
the mapping $\Ran(\widetilde{K}_0(\cdot))$ is smooth, 
that is, $\zeta_K\vert_{Z_0}\colon Z_0\to\Gr(\Hc^K)$ is smooth. 
Since $Z_0$ is a suitably neighborhood of the arbitrary point $s_0\in Z$, the proof is complete. 
\end{example}

We refer to Proposition~\ref{smoothomog} for examples of 
admissible reproducing kernels on Hermitian bundles with infinite-dimensional fibers,  
however we will briefly discuss right now the simplest instance of such a kernel, 
namely the universal reproducing kernel (cf. \cite{BG11}).  
It lives on the universal bundle of a complex Hilbert space,  
which is a basic example of Hermitian vector bundle.

\begin{example}\label{univker}
 \normalfont
If $\Hc$ is a complex Hilbert space, then the universal bundle $\Pi_{\Hc}$ has a natural Hermitian structure 
given by 
$(x\mid y)_{\Sc}:=(x\mid y)_{\Hc}$ for all $\Sc\in\Gr(\Hc)$ and $x,y\in \Sc$.
This Hermitian bundle carries a natural reproducing kernel $Q_{\Hc}$ defined by
$$
Q_{\Hc}(\Sc_1,\Sc_2):=p_{\Sc_1}\vert_{\Sc_2}\colon \Sc_2\to \Sc_1\quad \text{ for }\Sc_1,\Sc_2\in \Gr(\Hc).
$$
Fix an element $\Sc_0\in\Gr(\Hc)$. 
Then by restriction we obtain the Hermitian vector bundle  
$\Pi_{\Sc_0}\colon \Tc_{\Sc_0}(\Hc)\to\Gr_{\Sc_0}(\Hc)$ as a subbundle of $\Pi_{\Hc}$.
Denote by $Q_{\Sc_0}$ the restriction of the kernel $Q_{\Hc}$ to the bundle $\Pi_{\Sc_0}$. 
For every $\Sc\in\Gr_{\Sc_0}(\Hc)$ there exists $u\in\Uc$ (see Example~\ref{ex4}) 
such that 
$u\Sc_0=\Sc$ and $u\Sc_0^\perp=\Sc^\perp$. 
Then $up_{\Sc_0}=p_{\Sc}u$, that is, 
$p_{\Sc}=up_{\Sc_0}u^{-1}$. 
Thus for all $u_1,u_2\in\Uc$ and $x_1,x_2\in\Sc_0$ 
we have 
$$
Q_{\Sc_0}(u_1\Sc_0,u_2\Sc_0)(u_2x_2)
=p_{u_1\Sc_0}(u_2x_2)=u_1p_{\Sc_0}(u_1^{-1}u_2x_2).
$$ 
See \cite[Def. 4.2 and Rem. 4.3]{BG11} for some more details.
\end{example}

Let $\Pi\colon D\to Z$ and $\widetilde{\Pi}\colon \widetilde{D}\to\widetilde{Z}$  
be Hermitian vector bundles. 
A {\it quasimorphism} 
of $\Pi$ into $\widetilde{\Pi}$ is a pair 
 $\Theta=(\delta,\zeta)$,  
where 
$\delta\colon D\to\widetilde{D}$ and 
$\zeta\colon Z\to\widetilde{Z}$ 
are (not necessarily smooth) mappings such that:
\begin{itemize}
\item[(i)] $\zeta\circ\Pi=\widetilde{\Pi}\circ\delta$; 
\item[(ii)] for every $z\in Z$ the mapping 
$\delta_z:=\delta|_{D_z}\colon D_z\to\wt D_{\zeta(z)}$ 
is a bounded linear operator. 
\end{itemize}

\begin{definition}[\cite{BG11}]\label{involanti} 
\normalfont 
Let $\Pi\colon D\to Z$ and $\widetilde{\Pi}\colon \widetilde{D}\to\widetilde{Z}$ be 
Hermitian vector bundles with a quasimorphism $\Theta=(\delta,\zeta)$ from 
$\Pi$ to $\widetilde{\Pi}$. 
Assume that $\widetilde K$ is a reproducing kernel on $\widetilde{\Pi}$.
The {\it pull-back of the reproducing kernel} $\widetilde K$ through $\Theta$ 
is the reproducing kernel $\Theta^*{\widetilde K}$ on $\Pi$ defined by
\begin{equation}\label{pullbackrep}
(\forall s,t\in Z)\quad \Theta^*{\widetilde K}(s,t)
=\delta_s^*\circ {\widetilde K}(\zeta(s),\zeta(t))\circ\delta_t.
\end{equation}
\end{definition}

For later use, we now recall from \cite{BG11} the universality theorem for reproducing kernels.

\begin{theorem}\label{univinvol}
Let $\Pi\colon D\to Z$ be a Hermitian vector bundle 
endowed with a reproducing kernel $K$. 
If we define $\delta_K:=(\zeta_K\circ\Pi,\widehat{K})\colon D\to\Tc(\Hc^K)$, 
then we have the vector bundle quasimorphism $\Delta_K:=(\delta_K,\zeta_K)$ 
from
$\Pi$ into the universal bundle 
$\Pi_{\Hc^K}$ and moreover 
$K=(\Delta_K)^*Q_{\Hc^K}$. 
\end{theorem}

\begin{proof}
See \cite[Ths. 5.1 and 6.2]{BG11}. 
\end{proof}

We will call the quasimorphism  $\Delta_K$ constructed in Theorem~\ref{univinvol} 
{\it the classifying quasimorphism} associated with the kernel $K$. 
In order to define the notion of linear connection induced by a reproducing kernel, 
we need to elucidate when the first component of a classifying quasimorphism is a fiberwise isomorphism. 
This is done in the next subsection.

\subsection{Quantization maps and kernels}\label{Sect5}

Motivated by the significant physical interpretation given in \cite{Od88} and \cite{Od92} 
(see also \cite{MP97} and \cite{BG11}) to maps from manifolds  
into the projective space of a complex Hilbert space, we use the following terminology.

\begin{definition}\label{Quanmap}
\normalfont
Let $Z$ be a Banach manifold and and $\Hc$ be a complex Hilbert space. 
Any smooth mapping $\zeta\colon Z\to\Gr(\Hc)$ 
is termed a {\it quantization map} from $Z$ to $\Hc$.
\end{definition}

In the framework of Definition~\ref{Quanmap}, set
$$
\Dc_\zeta:=\{(s,x)\in Z\times\Hc: x\in\zeta(s)\}.
$$
Then $\Dc_\zeta$ is a Banach manifold and the projection
$$
\Pi^\zeta\colon (s,x)\mapsto s,\quad \Dc_\zeta\to Z
$$
defines a vector bundle, with local trivializations
$$
(\Pi^\zeta)^{-1}(\Omega_s)\simeq \Omega_s\times\zeta(s)$$ 
for suitably small open subsets 
$\Omega_s\subseteq \zeta^{-1}(\Gr_{\zeta(s)}(\Hc))$, 
where $s\in Z$ and the fiber at $s\in Z$ is identified to $\zeta(s)$. 
Put now
$$
\psi_\zeta(s,x):=(\zeta(s),x), \quad s\in Z, x\in\zeta(s)\subseteq\Hc,
$$
so that $(\psi_\zeta,\zeta)$ is a vector bundle morphism from
$\Pi_\zeta\colon\Dc_\zeta\to Z$ to the universal bundle $\Tc(\Hc)\to\Gr(\Hc)$. 
In fact, by identifying $\Dc_\zeta$ with $\{(s,(\zeta(s),x))\mid s\in Z, x\in\zeta(s)\}$ 
one has that $\Pi_\zeta\colon\Dc_\zeta\to Z$ is isomorphic to the pull-back of the universal bundle 
$\Tc(\Hc)\to\Gr(\Hc)$ through the mapping $\zeta$.

We provide the bundle
$\Pi_\zeta\colon \Dc_\zeta\to Z$ with the Hermitian structure induced from $\Hc$ 
and with the reproducing kernel given by
$$
K_{\zeta}(s,t):=p_{\zeta(s)}|_{\zeta(t)}\colon\zeta(t)\to\zeta(s)\quad (s,t\in Z).
$$
Clearly, $K_{\zeta}(s,s)=\id_{\zeta(s)}$ for every $s\in Z$. 
In fact, $K_{\zeta}$ is admissible. 
Firstly, $\zeta$ is smooth by assumption. 
Moreover, for every $\Sc_0\in\Gr(\Hc)$ the restriction $K_{\zeta}$ on $\zeta^{-1}(\Gr_{\Sc_0}(\Hc))$ 
can be seen as $Q_{\Hc,\Sc_0}\circ(\zeta,\zeta)$ and then we can apply to $Q_{\Hc,\Sc_0}$ 
the same argument of part (a) in the proof of Proposition \ref{smoothomog} below 
to deduce that $K_{\zeta}$ is smooth. 

Thus to every quantization map $\zeta$ there corresponds an admissible reproducing kernel $K_{\zeta}$, 
and it is natural to investigate the correspondence in the opposite direction.

In the following result the bundle 
$\Pi_{\zeta_K}\colon\Dc_{\zeta_{K}}\to Z$ is endowed with the Hermitian structure and the reproducing kernel 
$K_{\zeta_K}$ introduced above.

\begin{theorem}\label{isoquantum}
Let $\Pi\colon D\to Z$ be a Hermitian vector bundle with an admissible reproducing kernel $K$.
Then the following assertions hold. 
\begin{itemize}
\item[{\rm(i)}] The mapping $\check K:=(\Pi,\widehat K)\colon D\to\Dc_{\zeta_{K}}$ is a diffeomorphism 
and we have the commutative diagram   
$$
\begin{CD}
D @>{\check K}>> \Dc_{\zeta_K} \\
@V{\Pi}VV @VV{\Pi^{\zeta_K}}V \\
Z @>{\id_Z}>> Z 
\end{CD}
$$
which gives an isomorphism $\Delta_{\zeta_K}:=(\check K,\id_Z)$ of smooth vector bundles from  
$\Pi$ onto the pull-back of the tautological bundle $\Pi_{\Hc^K}$ through $\zeta_K$. 
\item[{\rm(ii)}] The quasimorphism $\Delta_K=(\delta_K,\zeta_K)$ of Theorem \ref{univinvol} 
is smooth and factorizes according to the commutative diagram 
$$
\begin{CD}
\delta_K\colon D @>{\check{K}}>> \Dc_{\zeta_K} @>{\psi_K}>> \Tc(\Hc^K) \\
\hskip20pt @VV{\Pi}V @VV{\Pi_{\zeta_K}}V @VV{\Pi_{\Hc^K}}V \\
\zeta_K\colon Z @>{\id_Z}>> Z @>{\zeta_K}>> \Gr(\Hc^K),
\end{CD}
$$
where $\psi_K:=\psi_{\zeta_K}$ is as after Definition \ref{Quanmap}. 
\item[{\rm(iii)}]  The pull-back relation 
$K=\Delta_{K}^*Q_{\Hc^K}$ factorizes as
$$
K=\Delta_{\zeta_K}^*K_{\zeta_K}
=\Delta_{\zeta_K}^*(\psi_K,\zeta_K)^*Q_{\Hc^K}
=\Delta_{K}^*Q_{\Hc^K}.
$$
\end{itemize}
\end{theorem}

\begin{proof}
(i) It follows by Lemma~\ref{smoothhat} that $\widehat{K}$ is smooth, 
hence also $\check K$ is smooth. 

To prove that $\check K$ is bijective, 
first suppose that $\xi,\eta\in D$ and $\check K(\xi)=\check K(\eta)$. 
This means that
$\Pi(\xi)=s=\Pi(\eta)$  with $s\in Z$, and therefore $\xi,\eta\in D_s$. 
Since $\widehat K$ is injective on $D_s$ by Lemma~\ref{equivKinverse}, 
we deduce that $\xi=\eta$. 
Now, take an element $(s,x)$ with
$s\in Z$ and $x\in\zeta_K(s)=\widehat K(D_s)$, where the equality of these two subsets holds because
$\widehat K(D_s)$ is closed by Lemma~\ref{equivKinverse} again. 
Thus there exists
$\xi\in D_s$ such that $x=\widehat K(\xi)$ and therefore
$(s,x)=(\Pi(\xi),\widehat K(\xi))$. 
We have proved that $\check K$ is both injective and surjective. 
In conclusion, $\check K$ is a bijection between the bundles $D$ and 
$\Dc_{\zeta_K}$, 
and the fact that $\Delta_{\zeta_K}$ is a vector bundle morphism follows readily. 

Moreover, since 
$\widehat K\colon D_s\to\widehat K(D_s)$ is continuous for every $s\in Z$, 
it follows by Lemma~\ref{equivKinverse} that we can use the open mapping theorem to obtain that 
$\widehat K$ is a fiberwise topological isomorphism. 
Since $K$ is an admissible reproducing kernel, it follows 
that the mapping 
$\zeta_K\colon D\to\Gr(\Hc^K)$ is smooth, and then the discussion after Definition~\ref{Quanmap} 
provides local trivializations for the bundle $\Pi^{\zeta_K}\colon\Dc_{\zeta_K}\to Z$. 
It then follows that $\check K$ is represented locally (as in \cite[Ch. III, \S 1, VB~Mor~2]{La01}) by a smooth mapping 
with values invertible operators on the typical fiber, 
and then its pointwise inverse is also smooth.
This shows that the inverse mapping of the bijection $\check K$ is also smooth, 
hence $\check K$ is a diffeomorphism.

(ii) Straightforward consequence of (i), since $\psi_K\circ(\Pi,\widehat K)=(\zeta_K,\widehat K)=\delta_K$.

(iii) The proof is similar to the one of Theorem~\ref{univinvol},   
by also using  
$(\psi_K,\zeta_K)\circ\Delta_{\zeta_K}=(\psi_K,\zeta_K)\circ((\Pi,\widehat K),\id_Z)$ 
and 
$(\psi_K\circ(\Pi,\widehat K),\zeta_K)=(\delta_K,\zeta_K)=\Delta_K$. 
\end{proof}

\section{Connections associated with reproducing kernels}\label{sectconnect}

\subsection{Linear connections induced by reproducing kernels}

Let $\Pi\colon D\to Z$ be a Hermitian vector bundle endowed with an admissible  
reproducing kernel $K$. 
Let $\Delta_K=(\delta_K,\zeta_K)$ be the classifying quasimorphism for $K$ 
constructed in Theorem~\ref{univinvol}, which is smooth by Theorem~\ref{isoquantum} 
since $K$ is supposed to be admissible. 
Assume for a moment that $\Sc_0$ in $\Gr(\Hc^K)$ is such that 
$\zeta_K(Z)\subseteq\Gr_{\Sc_0}(\Hc^K)$ (it holds for instance if $Z$ is connected), 
and therefore $\delta_K(D)\subseteq\Tc_{\Sc_0}(\Hc^K)$, 
so $\Delta_K$ is a morphism from $\Pi$ to the universal bundle $\Pi_{\Sc_0}$ at $\Sc_0\subseteq\Hc^K$:
$$
\begin{CD}
D @>{\delta_K}>> \Tc_{\Sc_0}(\Hc^K)\\ 
@V{\Pi}VV @VV{\Pi_{\Hc^K,\Sc_0}}V \\ 
Z @>{\zeta_K}>> \Gr_{\Sc_0}(\Hc^K) \\ 
\end{CD}
$$
Let $E_p$ be the conditional expectation naturally associated to 
the orthogonal projection $p:=p_{\Sc_0}\colon\Hc^K\to\Sc_0$. 
Then let
$\Phi_{\Sc_0}$ denote the connection induced by $E_p$ on the bundle $\Pi_{\Hc^K,\Sc_0}$ 
as in Definition \ref{tautoconn}, i.e., 
$$
\Phi_{\Sc_0}\colon[((u,X),(x,y))]\mapsto[((u,0),(x,E_p(X)x+y))].
$$
Since we are assuming that $K(s,s)$ is invertible on $D_s$ for all $s\in Z$, 
we have that the map $\delta_K$ is a fiberwise linear 
isomorphism from $D_s$ onto $\widehat K(D_s)$  
and then the following definition is consistent, according to 
Proposition~\ref{prop13}.

\begin{definition}\label{pullconn}
\normalfont
Under the above conditions, we call  {\it connection induced by the admissible reproducing kernel} $K$ 
the pull-back connection $\Phi_K$ on $\Pi$ given by
$$
\Phi_K:=(\Delta_K)^*(\Phi_{\Sc_0}).
$$
Note that condition $\zeta_K(Z)\subseteq\Gr_{\Sc_0}(\Hc^K)$, as prior to the definition, 
may always be assumed without loss of generality since, otherwise, 
one can consider partitioning $D$ into the (open) submanifolds 
$\zeta_K^{-1}(\Gr_{\Sc_0}(\Hc^K))$, define the connection on each of them, 
and then define the global connection on $D$ piecewise.  
\end{definition}

We compute now the covariant derivative 
for the connection induced by a reproducing kernel in the framework of Definition~\ref{pullconn}.

\begin{theorem}\label{deriv_prop3}
In the setting of Definition~\ref{pullconn},  
let $\nabla_K\colon\Omega^0(Z,D)\to\Omega^1(Z,D)$ be the covariant derivative 
for the connection induced by~$K$. 

If $\sigma\in\Omega^0(Z,D)$ has the property that there exists 
$\widetilde{\sigma}\in\Omega^0(\Gr_{\Sc_0}(\Hc^K),\Tc_{\Sc_0}(\Hc^K))$ 
such that $\delta_K\circ\sigma=\widetilde{\sigma}\circ\zeta_K$, 
then for $s\in Z$ and $X\in T_sZ$ we have 
$$
(\nabla\sigma)(X)=
K(s,s)^{-1}
\bigl(\underbrace{\bigl((p_{\zeta_K(s)}(\de(\widehat{K}\circ\sigma)(X)))\bigr)}_{\hskip50pt
\in\Hc^K\subseteq\Omega^0(Z,D)}(s)\bigr).
$$ 
\end{theorem}

An equivalent way of expressing the conclusion of Theorem~\ref{deriv_prop3} 
is that for $s\in Z$, $t_0>0$, $\gamma\in\Ci((-t_0,t_0),Z)$ with $\gamma(0)=s$ 
and $\sigma\in\Omega^0(Z,D)$ we have the formula 
$$
(\nabla\sigma)(\dot\gamma(0))
=K(s,s)^{-1}\bigl(p_{\zeta_K(s)}\Bigl(\frac{\de}{\de t}
\Big\vert_{t=0}\widehat{K}(\sigma(\gamma(t)))\Bigr)(s)\bigr)
$$
which only requires the derivative at $t=0$ 
of the function $\widehat{K}\circ\sigma\circ\gamma\colon(-t_0,t_0)\to\Hc^K$ 
and then to take the orthogonal projection of the derivative on the subspace $\zeta_K(s)$ of~$\Hc^K$. 

\begin{proof}[Proof of Theorem~\ref{deriv_prop3}]
Recall that for every $\xi\in D$ we have
$$\widehat{K}(\xi)=K_\xi
=K(\cdot,\Pi(\xi))\xi\text{ and }\delta_K(\xi)=(\zeta_K(\Pi(\xi)),\widehat{K}(\xi)).$$ 
Let $s\in Z$ and $X\in T_sZ$ arbitrary.

Since $\delta_K\circ\sigma=\widetilde{\sigma}\circ\zeta_K$, 
it follows by Proposition~\ref{deriv_prop1} that 
$\delta_K\circ\nabla\sigma=\widetilde{\nabla}\widetilde{\sigma}\circ T(\zeta_K)$, 
where $\widetilde{\nabla}$ denotes the covariant derivative 
for the universal connection on the tautological vector bundle 
$\Pi_{\Hc,\Sc_0}\colon\Tc_{\Sc_0}(\Hc^K)\to\Gr_{\Sc_0}(\Hc^K)$. 
In particular  
\begin{equation}\label{deriv_prop3_proof_eq1}
(\zeta_K(s),\widehat{K}((\nabla\sigma)(X)))
=\delta_K((\nabla\sigma)(X))
=\widetilde{\nabla}\widetilde{\sigma}(T(\zeta_K)X).
\end{equation}
On the other hand, since $\widetilde{\sigma}\in\Omega^0(\Gr_{\Sc_0}(\Hc^K),\Tc_{\Sc_0}(\Hc^K))$, 
there exists a uniquely determined function  $F_{\widetilde{\sigma}}\in\Ci(\Gr_{\Sc_0}(\Hc^K),\Hc^K)$ 
with $\widetilde{\sigma}(\cdot)=(\cdot,F_{\widetilde{\sigma}}(\cdot))$. 
Then by Proposition~\ref{univ_deriv} we obtain 
\begin{equation}\label{deriv_prop3_proof_eq2}
\widetilde{\nabla}\widetilde{\sigma}(T(\zeta_K)X)
=(\zeta_K(s),p_{\zeta_K(s)}(\de F_{\widetilde{\sigma}}(T(\zeta_K)X))).
\end{equation}
By using $\delta_K\circ\sigma=\widetilde{\sigma}\circ\zeta_K$ again, 
we obtain  
$F_{\widetilde{\sigma}}\circ\zeta_K=\widehat{K}\circ\sigma\colon Z\to\Hc^K$, 
hence by differentiation we obtain 
\begin{equation}\label{deriv_prop3_proof_eq3}
\de F_{\widetilde{\sigma}}\circ T(\zeta_K)=\de(\widehat{K}\circ\sigma).
\end{equation}
It now follows by \eqref{deriv_prop3_proof_eq1}--\eqref{deriv_prop3_proof_eq3} 
that 
$$
\widehat{K}((\nabla\sigma)(X))=p_{\zeta_K(s)}(\de(\widehat{K}\circ\sigma)(X))\in\Hc^K.
$$
Both sides of the above equality  are sections in the bundle $\Pi\colon D\to Z$, 
and moreover $(\nabla\sigma)X\in D_{s}$. 
By evaluating the left-hand side at the point $s\in Z$ we obtain  
the value $K(s,s)\bigl((\nabla\sigma)X)\in D_s$. 
Hence by evaluating both sides of the above equality at $s$ 
and then applying the operator $K(s,s)^{-1}$ to both sides of the equality 
obtained after the evaluation we obtain  
$(\nabla\sigma)(X)=K(s,s)^{-1}
\bigl((p_{\zeta_K(s)}(\de(\widehat{K}\circ\sigma)(X)))(s)\bigr)$,  
as we wanted to show.
\end{proof}

\subsection{Categorial aspects}

We will now discuss the functorial features of the above correspondence between reproducing kernels and linear connections, 
and it will follow that this correspondence is unique in a quite natural way.  
The precise statement actually concerns the relationship between various categories:  
\begin{itemize}
\item $\textbf{Hilb}$ is the category whose objects are the complex Hilbert spaces and the morphisms are the linear isometries. 
\item $\textbf{Herm}$ is the category whose objects are the Hermitian vector bundles and the morphisms are the 
bundle morphisms which are fiberwise unitary operators; 
\item $\textbf{Kernh}$ is the category whose objects are the admissible reproducing kernels on Hermitian bundles. 
The morphisms of this category are defined by 
$$
\Hom_{\textbf{Kernh}}(K_1,K_2)=\{\Theta\in\Hom_{\textbf{Herm}}(\Pi_1,\Pi_2)\mid \Theta^*(K_2)=K_1\}
$$ 
whenever $K_j$ is an admissible reproducing kernel on the Hermitian vector bundle~$\Pi_j$ for $j=1,2$. 
The morphisms $\Theta=(\delta,\zeta)$ in $\Hom_{\textbf{Kernh}}(K_1,K_2)$ satisfy that 
$\delta$ is a fiberwise diffeomorphism. This follows from the identity 
$K_1(t,t)=\delta^*\circ K_2(\zeta(t),\zeta(t))\circ\delta$, $t\in Z_1$, where 
$\Pi_1\colon D_1\to Z_1$, since $K_1, K_2$ are admissible.
\item $\textbf{LinConnect}$ is the category whose objects are the linear connections on Hermitian vector bundles 
and the morphisms are defined by 
$$
\text{Hom}_{\textbf{LinConnect}}(\Phi_1,\Phi_2)
=\{\Theta\in\Hom_{\textbf{Herm}}(\Pi_1,\Pi_2)\mid \Theta^*(\Phi_2)=\Phi_1\}
$$
whenever $\Phi_j$ is a linear connection on the Hermitian vector bundle~$\Pi_j$ for $j=1,2$. 
Note that a morphism $\Theta$ in $\text{Hom}_{\textbf{LinConnect}}(\Phi_1,\Phi_2)$ 
must be fiberwise diffeomorphic for the condition $\Theta^*(\Phi_2)=\Phi_1$ to make sense.    
\item $\Qc\colon\Hilb\to\textbf{Kernh}$ is the functor that constructs the universal reproducing kernel 
on the tautological bundle for a given Hilbert space.  
\item $\Phi\colon\Hilb\to\textbf{LinConnect}$ is the functor that constructs 
the universal connection on the tautological bundle for a given Hilbert space.  
\item ${\mathbb F}$ are the forgetful functors that associate to every kernel or connection 
the bundle where these objects are living.    
\item And finally, ${\mathbb A}\colon\textbf{Kernh}\to\textbf{LinConnect}$ 
is the functor defined by means of Definition~\ref{pullconn} 
on the level of objects of these categories and which acts identically on the level of morphisms. 
\end{itemize}

Here is the categorial characterization of the functor ${\mathbb A}$ 
from the category of the admissible reproducing kernels to the one of 
linear connections on Hermitian vector bundles.

\begin{theorem}\label{canonical}
There exists a unique functor from $\textbf{Kernh}$ into $\textbf{LinConnect}$ such that 
the diagram 
$$
\xymatrix{
& \textbf{LinConnect} \ar[dr]^{\mathbb F} & \\
\textbf{Hilb} \ar[ur]^{\Phi}\ar
[dr]^{\Qc}  & & \textbf{Herm} \\
& \textbf{Kernh}  \ar@{-->}[uu] \ar[ur]_{\mathbb F} &}
$$
is commutative, and that functor is ${\mathbb A}\colon\textbf{Kernh}\to\textbf{LinConnect}$. 
\end{theorem}

\begin{proof}
We first check that the diagram in the statement is commutative if 
the role of the dotted arrow is played by the functor~${\mathbb A}$. 
In fact, we have ${\mathbb F}\circ{\mathbb A}={\mathbb F}$ since 
the functor~${\mathbb A}$ takes an admissible reproducing kernel 
on some Hermitian vector bundle to a linear connection 
on the same Hermitian vector bundle. 
Moreover, it follows directly by Definition~\ref{pullconn} 
(see also \cite[Prop. 4.1]{BG08}, \cite[Prop. 4.5 and Ex.]{BG09}) that ${\mathbb A}\circ\Qc=\Phi$. 

To prove the uniqueness assertion, let us assume that 
${\mathbb B}\colon\textbf{Kernh}\to\textbf{LinConnect}$ 
is a functor such that ${\mathbb F}\circ{\mathbb B}={\mathbb F}$ and 
${\mathbb B}\circ\Qc=\Phi$. 
The latter equality shows that for every Hilbert space $\Hc$ 
we have ${\mathbb B}(Q_{\Hc})=\Phi_{\Hc}$, 
hence ${\mathbb B}(Q_{\Hc})={\mathbb A}(Q_{\Hc})$. 
Thus the functors ${\mathbb B}$ and ${\mathbb A}$ 
agree on the universal reproducing kernels. 

Now let $K$ be an arbitrary admissible reproducing kernel 
on a Hermitian vector bundle $\Pi$. 
The classifying morphism $\Delta_K\in\Hom_{\textbf{Herm}}(\Pi,\Tc(\Hc^K))$ 
has the property $\Delta_K^*(Q_{\Hc^K})=K$, 
hence we have $\Delta_K\in\Hom_{\textbf{Kernh}}(K,Q_{\Hc^K})$. 
By using the functor ${\mathbb B}\colon\textbf{Kernh}\to\textbf{LinConnect}$, 
it follows  
${\mathbb B}(\Delta_K)\in
\Hom_{\textbf{LinConnect}}({\mathbb B}(K),{\mathbb B}(Q_{\Hc^K}))$. 
On the other hand, by using the equality ${\mathbb F}\circ{\mathbb B}={\mathbb F}$ 
on morphisms in the category $\textbf{Kernh}$, we get ${\mathbb B}(\Delta_K)=\Delta_K$; 
moreover, we established above that 
${\mathbb B}(Q_{\Hc})=\Phi_{\Hc}$ for every 
Hilbert space $\Hc$, hence in particular ${\mathbb B}(Q_{\Hc^K})=\Phi_{\Hc^K}$. 
We thus obtain $\Delta_K\in\Hom_{\textbf{LinConnect}}({\mathbb B}(K),\Phi_{\Hc^K})$. 
By the definition of the morphisms in the category $\textbf{LinConnect}$, 
this means that ${\mathbb B}(K)=\Delta_K^*(\Phi_{\Hc^K})$, 
hence we have ${\mathbb B}(K)={\mathbb A}(K)$. 
Thus the functors ${\mathbb B}$ and ${\mathbb A}$ agree on the level of objects 
in the category $\textbf{Kernh}$. 

Furthermore, it follows by the condition ${\mathbb F}\circ{\mathbb B}={\mathbb F}$ 
that the functor~${\mathbb B}$ acts identically on the morphisms 
of the category $\textbf{Kernh}$, just as the functor ${\mathbb A}$ does. 
Thus eventually ${\mathbb B}={\mathbb A}$. 
\end{proof}

\section{Examples}\label{examples}

We will discuss here the linear connections associated with three types of examples, namely  
the usual operator-valued reproducing kernels (Subsection~\ref{ktb}), 
then the reproducing kernels on homogeneous vector bundles 
that occur in the geometric representation theory of Banach-Lie groups (Subsection~\ref{subsect_homog}), 
and finally the reproducing kernels related to the dilation theory of completely positive mappings (Subsection~\ref{subsect_cp}). 

\subsection{Reproducing kernels on trivial bundles}\label{ktb}
We illustrate here the theory established in the preceding sections by giving some results involving classical reproducing kernels on trivial vector bundles. 

\medskip
(a) \textit{General case.}

Let $\Xc$ be a set and $\Vc$ be a complex Hilbert space. 
Assume that  $\kappa\colon \Xc\times \Xc\to\Bc(\Vc)$ is the reproducing kernel 
of a Hilbert space denoted by $\Hc^\kappa$. 
This means in particular that, for every $x_i\in \Xc$, $v_i\in \Vc$, $i=1,\dots,n$,
$$
\sum_{i,j=1}^n(\kappa(x_i,x_j)v_j\mid v_i)_{\Vc}\ge0,
$$
and that $\Hc^\kappa$ is the Hilbert space of $\Vc$-valued functions on $\Xc$ generated by the space 
$\hbox{span}\{\kappa_x\otimes v:x\in \Xc, v\in \Vc\}$, where
$$
\kappa_x\otimes v:=\kappa(\cdot,x)v\colon \Xc\to \Vc,
$$
with respect to the inner product given by 
$$
(\kappa_x\otimes v\mid\kappa_y\otimes w)_{\Hc^\kappa}
:=(\kappa(y,x)v\mid w)_\Vc,
$$ 
see \cite[Theorem I.1.4, (2) and (a)]{Ne00}. 

In the sequel we assume that $\Xc$ is a Banach manifold, with $\kappa$ smooth, such that $\kappa(x,x)$ is invertible in $\Bc(\Vc)$ for all $x$,  which corresponds to a reproducing kernel 
on the trivial Hermitian bundle $\Xc\times\Vc\to\Xc$.
In this special case we now compute the covariant derivative for the connection induced by the reproducing kernel $K$ 
 when moreover $\dim\Vc=1$, hence we may assume $\Vc=\C$. 
 In the following statement we use subscripts to denote the values of differential 1-forms on $\Xc$.

\begin{proposition}\label{codetri}
For every smooth section $\sigma(\cdot)=(\cdot,F_{\sigma}(\cdot))$ 
of the trivial bundle $\Xc\times\C\to\Xc$, 
where $F_{\sigma}\in\Ci(\Xc,\C)$, we have 
$$
(\forall x\in\Xc)\quad (\nabla\sigma)_x
=(x,(\de F_{\sigma})_x
+F_{\sigma}(x)\frac{\partial_2\kappa(x,x)}{\kappa(x,x)}
\in\{x\}\times\Bc(T_x\Xc,\C).
$$
Hence the covariant derivative $\nabla$ can be identified with 
the first order linear differential operator $\nabla\colon\Omega^0(\Xc,\C)\to\Omega^1(\Xc,\C)$ 
defined by 
\begin{equation}\label{explicit}
(\forall F\in\Ci(\Xc,\C))\quad \nabla F=\de F+\alpha_\kappa\cdot F 
\end{equation} 
where the differential 1-form $\alpha_\kappa\in\Omega^1(\Xc,\C)$ 
is defined by $\displaystyle(\alpha_\kappa)_x=\frac{\partial_2\kappa(x,x)}{\kappa(x,x)}$ for all $x\in\Xc$. 
\end{proposition}

\begin{proof}
For arbitrary $x\in \Xc$ we have $\zeta_K(x)=\C\cdot\kappa_x$, where $\kappa_x=\kappa(\cdot,x)\in\Hc^\kappa$ 
and $(\kappa_x\mid\kappa_x)_{\Hc^\kappa}=\kappa(x,x)$. 
Therefore 
$$
(\forall x\in\Xc)\quad p_{\zeta_K(x)}=\frac{(\cdot\mid\kappa_x)}{\kappa(x,x)}\kappa_x.
$$
Moreover 
$$
(\forall x\in\Xc)(\forall v\in\C)\quad \widehat{K}(x,v)=v\kappa_x=v\kappa(\cdot,x). $$
Now consider a smooth path $\gamma\in\Ci((-t_0,t_0),\Xc)$ with $\gamma(0)=x\in\Xc$ 
and a smooth section $\sigma(\cdot)=(\cdot,F_{\sigma}(\cdot))$, 
where $F_{\sigma}\in\Ci(\Xc,\C)$.  
For arbitrary $t\in(-t_0,t_0)$ we have 
$\widehat{K}(\sigma(\gamma(t)))
=\widehat{K}(\gamma(t),F_{\sigma}(\gamma(t)))
=F_{\sigma}(\gamma(t))\kappa(\cdot,\gamma(t))
=F_{\sigma}(\gamma(t))\kappa_{\gamma(t)}$ 
and 
$$
\begin{aligned}
p_{\zeta_K(x)}(\widehat{K}(\sigma(\gamma(t))))
&=\frac{(\widehat{K}(\sigma(\gamma(t)))\mid\kappa_x)}{\kappa(x,x)}\kappa_x
=\frac{(F_{\sigma}(\gamma(t))\kappa_{\gamma(t)}\mid\kappa_x)}{\kappa(x,x)}\kappa_x \\
&=F_{\sigma}(\gamma(t))\frac{\kappa(x,\gamma(t))}{\kappa(x,x)}\kappa_x
\end{aligned}
$$
hence 
$$
\frac{\de}{\de t}\Big\vert_{t=0}p_{\zeta_K(x)}(\widehat{K}(\sigma(\gamma(t))))
=\Bigl(\de F_{\sigma}(\dot\gamma(0))
+F_{\sigma}(x)\frac{(\partial_2\kappa(x,x))(\dot\gamma(0))}{\kappa(x,x)}
\Bigr)\kappa_x. 
$$
Therefore, by using Theorem~\ref{deriv_prop3} 
it follows that for any smooth section $\sigma(\cdot)=(\cdot,F_{\sigma}(\cdot))$, 
where $F_{\sigma}\in\Ci(Z,\C)$, we have 
$$
(\forall x\in\Xc)\quad (\nabla\sigma)_x
=(x,(\de F_{\sigma})_x
+F_{\sigma}(x)\frac{\partial_2\kappa(x,x)}{\kappa(x,x)}\in\{x\}\times\Bc(T_x\Xc,\C)
$$
as we wanted to show. 
\end{proof}

We next recall a few classical reproducing kernels on one-dimensional trivial vector bundles 
arising in function theory, 
so that their linear connections can be computed using the above formula~\eqref{explicit}.

(b)\textit{Classical (scalar-valued) reproducing kernels:} 
Bergman and Hardy spaces on the disk and the half-plane; Fock space.

(b.1) For $\nu>1$, the corresponding {\it Bergman space} is the Hilbert space 
$$
{\mathfrak B}^2_\nu(\D)=\Bigl\{f\in\Oc(\D){\Big\vert}
\frac{\nu-1}{\pi}\int_\D\vert f(z)\vert^2 (1-\vert z\vert^2)^{\nu-2}\,dz<\infty\Bigr\},
$$
where $dz$ is the Lebesgue measure on $\D$. 
Clearly, the polynomial functions belong to ${\mathfrak B}^2_\nu(\D)$.
In the \lq\lq limiting case" $\nu=1$ one obtains the {\it Hardy space}  
$$
{\mathfrak H}^2(\D)=\Bigl\{f\in\Oc(\D){\Big\vert}
\sup_{0<r<1}\frac{1}{2\pi}\int_0^{2\pi}\vert f(re^{i\theta})\vert^2\ d\theta<\infty\Bigr\}.
$$  
Let us denote for a while both the Bergman and Hardy spaces on the unit disk by  
the same symbol $\Hc_\nu(\D)$, $\nu\ge1$ (the Hardy space corresponds to $\nu=1$). 
These are reproducing kernel Hilbert spaces with kernels 
$$
K_\D^{(\nu)}(s,t)=\frac{1}{(1-\overline t s)^\nu} \quad (s,t\in\D;\nu\ge1); 
$$
see \cite{Hi08}, \cite{Ne00}, for instance.

In a similar way, for the upper halfplane $\Ub:=\{z\in\C:\Im z>0\}$, define the Bergman space 
$$
{\mathfrak B}^2_\nu(\Ub)=\Bigl\{F\in\Oc(\Ub){\Big\vert}
\frac{(\nu-1)}{\pi}\int_\D\vert F(z)\vert^2 (\Im z)^{\nu-2}\,dz<\infty\Bigr\}.
$$
It turns out that ${\mathfrak B}^2_\nu(\Ub)$ is unitarily equivalent to ${\mathfrak B}^2_\nu(\D)$ through the Cayley transform $\displaystyle\varphi(z):=\frac{z-i}{z+i}$, ($z\in\Ub$). 
The reproducing kernel of ${\mathfrak B}^2_\nu(\Ub)$ is given by
$$
K_\Ub^{(\nu)}(z,w):=\frac{1}{4}\frac{(2i)^\nu}{(z-\overline w)^\nu},  
\quad (z,w\in\Ub);
$$
see \cite[p. 16]{Hi08}. 
The Hardy space on $\Ub$,
$$ 
{\mathfrak H}^2(\Ub)=\Bigl\{F\in\Oc(\Ub){\Big\vert}
\sup_{y>0}\frac{1}{2\pi}\int_{-\infty}^\infty \vert F(x+iy)\vert^2 dx<\infty\Bigr\}, 
$$
is also a reproducing kernel Hilbert space with kernel
$$
K_\Ub^{(1)}(z,w):=\frac{i}{2(z-\overline w)},
\quad (z,w\in\Ub);
$$
see \cite[p. 19]{Hi08}.

Unlike the Bergman case, the image of ${\mathfrak H}^2(\Ub)$ via the unitary isomorphism induced by the Cayley transform $\varphi$ is not all of ${\mathfrak H}^2(\D)$, but just a closed proper subspace of it, see \cite[Ch. 8]{Ho62}. 
Despite this, we use the symbol $\Hc_\nu(\Ub)$ to refer both the Bergman and the Hardy spaces on $\Ub$. 
As in the unit disk case, the Hardy space corresponds to $\nu=1$.

We should note that in the above expressions of the kernels we omitted the number $1/\pi$ that occurs in the expressions given in \cite{Hi08}, since we preferred to keep it in the integral conditions defining the Bergman and Hardy spaces.

(b.2) Let $\Ec$ be a complex vector space and let $\beta\colon\Ec\times\Ec\to\C$ be a positive semidefinite hermitian form on $\Ec$. Then $K_{\Fc,\beta}(z,w):=e^{\beta(z,w)}$, for $z,w\in\Ec$, is a reproducing kernel which generates a Hilbert space $\Hc^{K_{\Fc,\beta}}$ denoted by  
$\Fc(\Ec,\beta):=\Hc^{K_{\Fc,\beta}}$, and which is called {\it Fock space} associated with $\Ec$ and $\beta$. 
See \cite[p. 38]{Ne00}.

When $\Ec=\C^n$ is finite-dimensional, the Fock space can be realized as 
$$
\Fc(\Ec,\beta)=\Bigl\{F\in\Oc(\Ec){\Big\vert}\frac{1}{\pi^n}\int_\Ec \vert F(z)\vert^2\, e^{-\beta(z,z)} dz<\infty\Bigr\}.
$$
See for instance \cite[pp. 25, 26]{Hi08}.

On account of the above examples, the former point (a) applies to:

(1) The set $\Xc:=\D$, the Hilbert spaces 
$\Vc:=\C$ and $\Hc^\kappa:=\Hc_\nu(\D)$, with the trivial bundle 
$\D\times\C\to\D$ and kernel $\kappa:=K_\D^{(\nu)}$.

(2) $\Xc:=\Ub$, 
$\Vc:=\C$ and $\Hc^\kappa:=\Hc_\nu(\Ub)$ , with trivial bundle 
$\Ub\times\C\to\Ub$ and kernel $\kappa:=K_\Ub^{(\nu)}$.

(3) $\Xc:=\Ec$, 
$\Vc:=\C$ and $\Hc^\kappa:=\Fc(\Ec,\beta)$ , with trivial bundle 
$\Ec\times\C\to\Ec$ and kernel $\kappa:=K_{\Fc,\beta}$.
(The prototypical example of Fock space occurs when $\Ec=\C^n$ and 
$\beta(z,w)=\sum_{j=1}^n\, z_j\overline{w_j}$, $z=(z_1,\dots,z_n), 
w=(w_1,\dots,w_n)\in\C^n$, $n\in\N$; see \cite[p. 10]{Ne00}.)

\smallbreak
 
It is straightforward to compute the first-order differential operator \eqref{explicit} 
in this case, in any of the above examples: 

(1) For $K_\D^{(\nu)}(s,t)=(1-\overline t s)^{-\nu}$ generating $\Hc_\nu(\D)$, $\nu\ge1$, we have 
$$
(\nabla\sigma)_s (z)=\de F_\sigma(z)-\frac{\nu\,s F_\sigma(s)}{(1-\vert s\vert^2)}\,\overline z
$$
for $\sigma\equiv F_\sigma\in\C^\infty(\D,\C)$, $s\in\D$, $z\in\C$.

(2) For $K_\Ub^{(\nu)}(z,w)=\frac{1}{4}(2i)^\nu(z-\overline w)^{-\nu}$ generating $\Hc_\nu(\Ub)$, $\nu\ge1$, we have 
$$
(\nabla\sigma)_z (\lambda)=\de F_\sigma(\lambda)-\frac{\nu F_\sigma(z)}{\Im z}\,\overline \lambda
$$
for $\sigma\equiv F_\sigma\in\C^\infty(\Ub,\C)$, $z\in\D$, $\lambda\in\C$.

(3) For $K_{\Fc,\beta}(z,w)=\exp({\sum_{j=1}^n\, z_j\overline{w_j}})$, $z,w\in\C^n$, generating the Fock space on $\C^n$, one gets
$$
(\nabla\sigma)_z (\lambda)=\de F_\sigma(\lambda)
+F_\sigma(z)\sum_{j=1}^n\, z_j\overline{\lambda_j}
$$
for $\sigma\equiv F_\sigma\in\C^\infty(\C^n,\C)$, 
$z=(z_j)_{j=1}^n, \lambda=(\lambda_j)_{j=1}^n\in\C^n$.

\subsection{Reproducing kernels on homogeneous vector bundles}\label{subsect_homog}
 
Let $G_A$ be a Banach-Lie group with a Banach-Lie subgroup $G_B$.
Let 
$\rho_A\colon G_A\to \Bc(\Hc_A)$ and $\rho_B\colon G_B\to \Bc(\Hc_B)$ 
be uniformly continuous unitary representations 
with 
$\Hc_B\subseteq\Hc_A$, $\rho_B(g)=\rho_A(g)|_{\Hc_B}$ for
$g\in G_B$ and
$\Hc_A=\overline{\spann}\rho_A(G_A)\Hc_B$.

Let us consider the homogeneous vector bundle 
$\Pi_\rho\colon G_A\times_{G_B}\Hc_B\to G_A/G_B$, induced by the representation $\rho_B$.  
We endow $\Pi_\rho$ with 
the 
Hermitian structure given by
$$ 
\left([(u,f)],[(u,h)]\right)_s:=(f\mid h)_{\Hc}, \quad u\in G_A,s:=u G_B, f,h\in\Hc_B.
$$
Let $P\colon\Hc_A\to\Hc_B$ be the orthogonal projection. 
We define
the reproducing kernel $K_\rho$ on the vector bundle
$\Pi_\rho\colon D=G_A\times_{G_B}\Hc_B\to G_A/G_B$ by
\begin{equation}\label{homoker}
K_\rho(uG_B,vG_B)[(v,f)]=[(u,P(\rho_A(u^{-1})\rho_A(v)f))],
\end{equation}
for $uG_B,vG_B\in D$ and $f\in\Hc_B$
(see \cite {BG08}). 
There exists
a unitary operator $W\colon\Hc^{K_\rho}\to\Hc_A$ such that
$W(K_\eta)=\pi_A(v)g$ whenever $\eta=[(v,g)]\in D$; 
see the end of the proof of \cite[Proposition 4.1]{BG08}.

\begin{proposition}\label{smoothomog}
In the above setting, $K_\rho$ is an admissible reproducing kernel. 
\end{proposition}

\begin{proof} 
We set $K:=K_\rho$ and will check that the conditions of Definition~\ref{adm_def} are satisfied. 

(a) The mapping 
$$G_A\times G_A\to\Bc(\Hc_B),\quad (u,v)\mapsto P\rho_A(u^{-1})\rho_A(v)\vert_{\Hc_B} $$ 
is smooth, hence the reproducing kernel $K_\rho$ is smooth.

(b) For all $s\in G_A/G_B$ we have $K(s,s)=\id_{D_s}$, hence $K(s,s)$ is invertible. 

(c) We have to prove that the mapping
$$
\zeta_K\colon G_A/G_B\to\Gr(\Hc^K),
\quad s\mapsto \widehat{K}(D_s)
$$
is smooth.
The unitary operator $W\colon\Hc^K\to\Hc_A$
induces a diffeomorphism
$$
\widetilde{W}\colon\Gr(\Hc^K)\to\Gr(\Hc_A),
\quad \Sc\mapsto W(\Sc),
$$
hence it will be enough to show that the mapping
$\widetilde{W}\circ\zeta_K\colon G_A/G_B\to\Gr(\Hc_A)$
is smooth.
To this end, note that for every
$s=uG_B\in G_A/G_B$ we have
$$
\widetilde{W}\circ\,\zeta_K(uG_B)
=W(\{K_{[(u,f)]}\mid [(u,f)]\in D_s\})
=\{\rho_A(u)f\mid f\in \Hc_B\} 
=\rho_A(u)\Hc_B
$$
hence there exists a commutative diagram
$$
\begin{CD}
G_A @>{\rho_A}>> \GL(\Hc_A)\\
@V{u\mapsto uG_B}VV  @VV{V\mapsto V(\Hc_B)}V \\
G_A/G_B @>{\widetilde{W}\circ\zeta_K}>> \Gr(\Hc_A)
\end{CD}
$$
whose vertical arrows are submersions.
It then follows by \cite[Cor. 8.4]{Up85}
that the mapping $\widetilde{W}\circ\zeta_K$ is smooth, and we are done.
\end{proof}

Using  $W$ as in the above proof we may suppose that $\Hc^{K_\rho}=\Hc_A$. 
Then for the classifying morphism of the admissible reproducing kernel $K_\rho$ 
we have $\Delta_{K_\rho}\equiv\Delta_\rho:=(\delta_\rho,\zeta_\rho)$, 
where   
$$
\begin{aligned}
\delta_\rho \colon G_A\times_{G_B}\Hc_B\to \Tc_{\Hc_B}(\Hc_A),
\quad &[(u,f)]\mapsto (\rho_A(u)\Hc_B,\rho_A(u)f), \\
\zeta_\rho \colon G_A/G_B \to \Gr_{\Hc_B}(\Hc_A),\quad & u G_B\mapsto \rho_A(u)\Hc_B.
\end{aligned} 
$$
Since $K_\rho$ is admissible, 
both these mappings are smooth by Theorem~\ref{isoquantum}.

In the following we will describe the linear connection associated with 
the admissible reproducing kernel $K_\rho$. 
In particular, we will show  how an application of the pull-back operation to 
the universal vector bundle induces a linear connection on a homogeneous vector bundle 
that may not be endowed with a reductive structure.   
It is straighforward to check that the map introduced in the following definition 
is indeed a linear connection on the homogeneous bundle 
$\Pi_\rho\colon G_A\times_{G_B}\Hc_B\to G_A/G_B$. 

\begin{definition}\label{homoconnection}
\normalfont
The \emph{natural connection} on the homogeneous Hermitian bundle $\Pi_\rho$ is the smooth mapping 
$\Phi_{\rho}\colon T(G_A\times_{G_B}\Hc_B)\to T(G_A\times_{G_B}\Hc_B)$, given for every $g\in G_A$, $X\in\Gg_A$ and 
$f,h\in\Hc_B$ by
$$
\Phi_{\rho}\colon[((g,X),(f,h))]\mapsto[((g,0),(f,P(\de\rho(X)f)+h))].
$$
\end{definition}

The above definition can also be derived from the classifying morphism~$\Delta_\rho$ of~$\Pi_\rho$, 
as the following result shows.   

\begin{proposition}\label{pullHomon}
For $\delta_{\rho}$ and $\zeta_{\rho}$ as above, 
the connection on the homogeneous bundle 
$G_A\times_{\G_B}\Hc_B\to G_A/G_B$ obtained as 
the pull-back $(\Delta_{\rho})^*(\Phi_{E_P})$ of the universal connection $\Phi_{E_P}$ 
through the morphism $\Delta_{\rho}=(\delta_{\rho},\zeta_{\rho})$ coincides with 
the connection~$\Phi_\rho$ above, 
$$\Phi_{\rho}=(\Delta_{\rho})^*(\Phi_{E_P}).$$
\end{proposition}

\begin{proof}
Recall from Theorem~\ref{conred} that $\Phi_{E_P}$ is given by 
$$[((u,Y),(f,h))]\mapsto[((u,0),(E_P(Y)f+h))]$$
for $u\in U(\Hc_A)$, $Y\in\Bc(\Hc_A)$ and $f,h\in\Hc_B$. 
On the other hand, the tangent map $T\delta_\rho$ is given by
$$
T\delta_\rho\colon[((g,X),(f,h))]\mapsto[((\rho(g),\de\rho(X)),(f,h))], 
$$
for $g\in G_A$, $X\in\Gg_A$, $f,h\in\Hc_B$. 
Thus the 
composition 
$\Phi=(T\delta_\rho)^{-1}\circ \Phi_{E_P}\circ T\delta_\rho$ is 
$$
\Phi\colon[((g,X),(f,h))]\mapsto[((g,0),(f,E_P(\de\rho(X))f+h))].
$$
Finally, since $E_P(\de \rho(X))f=P(\de\rho(X)f)$, we obtain that $\Phi=\Phi_{\rho}$ 
as claimed.
\end{proof}

\begin{corollary}\label{homoconinv}
For $\delta_{K_\rho}$ and $\zeta_{K_\rho}$ as above, 
the connection $\Phi_{K_\rho}$ associated with $K_\rho$ on the homogeneous vector bundle 
$\Pi\colon G_A\times_{\G_B}\Hc_B\to G_A/G_B$ in the sense of Definition~\ref{pullconn} 
coincides with the natural connection $\Phi_\rho$ on $\Pi$ given by Definition~\ref{homoconnection}. 
\end{corollary}

\begin{proof}
This is an immediate consequence of Proposition~\ref{pullHomon}, 
since both connections $\Phi_{K_\rho}$ and $\Phi_\rho$ are given by the fiberwise composition 
$\Phi=(T\delta_K)^{-1}\circ \Phi_{E_P}\circ T\delta_K$.
\end{proof}

We now turn to computing 
the covariant derivative associated with the preceding connection $\Phi_\rho$. 
In the case of a reductive structure one can use Remark~\ref{MC_rem}. 
For the general case note that if $\gg_A$ and $\gg_B$ are the Lie algebras of 
$G_A$ and $G_B$, respectively, then the adjoint action of $G_B$ on $\gg_A$ 
gives rise to a linear action on the quotient $\gg_A/\gg_B$ and 
we can then form the homogeneous vector bundle 
$G_A\times_{G_B}(\gg_A/\gg_B)$, 
which is isomorphic to the tangent bundle $T(G_A/G_B)$.

For any closed linear subspace $\mg$ of $\gg_A$ such that $\gg_A=\gg_B\dotplus\mg$ 
we have a linear topological isomorphism $\mg\simeq\gg_A/\gg_B$, 
which gives rise to a natural linear action of $G_B$ on $\mg$, 
hence to a homogeneous vector bundle $\G_A\times_{G_B}\mg$ which can be identified with $T(G_A/G_B)$. 
Note that such a subspace $\mg$ always exists since $G_B$ is a Banach-Lie subgroup of $G_A$, 
see \cite[Prop. 8.13]{Up85}.
In the special case $G_B=\{\1\}$ we get the identification $T(G_A)=G_A\ltimes\gg_A$ (see Remark \ref{conninduc}), 
so for every smooth function $\tilde\sigma\colon G_A\to\Hc_B$ 
we have $\de\tilde\sigma\colon G_A\ltimes\gg_A\to\Hc_B$. 

\begin{proposition}\label{homogcovderiv}
Let $\phi\colon G_A\to\Hc_B$ be smooth such that 
$\phi(uw)=\rho_A(w)^{-1}\phi(u)$ for all $u\in G_A$ and $w\in G_B$  
and define the corresponding smooth section   
$\sigma\colon G_A/G_B\to G_A\times_{G_B}\Hc_B$ by $\sigma(uG_B):=[(u,\phi(u))]$ 
for all $u\in G_A$. 
If there exists a smooth section 
$\widetilde\sigma\colon\Gr_{\Hc_B}(\Hc_A)\to\Tc_{\Hc_B}(\Hc_A)$ 
such that 
 $\widetilde\sigma(\rho_A(u)\Hc_B):=(\rho_A(u)\Hc_B,\rho_A(u)\phi(u))$ for all $u\in G_A$,  
then for every tangent vector $[(u,X)]\in G_A\times_{G_B}\mg=T(G_A/G_B)$ 
we have 
$$(\nabla\sigma)([(u,X)])=[(u,\de\phi(u,X)+P(\de\rho_A(X)\phi(u)))].$$ 
\end{proposition}

\begin{proof}
Denote $Z=G_A/G_B$, and let $s=u G_B\in Z$ and $X\in T_s Z$ arbitrary. 
Let $h\in\Hc_B$ such that 
$(\nabla\sigma)X:=(\nabla\sigma)([(u,X)])=[(u,h)]$. 
Since $\delta_\rho\circ\sigma=\widetilde{\sigma}\circ\zeta_\rho$ 
it follows by Propositions \ref{pullHomon}~and~\ref{deriv_prop1} that 
$\delta_\rho\circ\nabla\sigma=\widetilde{\nabla}\widetilde{\sigma}\circ T(\zeta_\rho)$, 
where $\widetilde{\nabla}$ denotes the covariant derivative for the universal connection on 
the tautological vector bundle 
$\Pi_{\Hc_A,\Hc_B}\colon\Tc_{\Hc_B}(\Hc_A)\to\Gr_{\Hc_B}(\Hc_A)$. 
In particular  
\begin{equation}\label{deriv_prop316_proof_eq1}
(\rho_A(u)\Hc_B,\rho_A(u)h)
=\delta_\rho((\nabla\sigma)X)
=\widetilde{\nabla}\widetilde{\sigma}(T(\zeta_\rho)X).
\end{equation}
Let $F_{\widetilde\sigma}\in\Ci(\Gr_{\Hc_B}(\Hc_A),\Hc_A)$ such that 
$F_{\widetilde\sigma}(\rho_A(u)\Hc_B)=\rho_A(u)\phi(u)$ for all $u\in G_A$, so that 
$\widetilde{\sigma}(\cdot)=(\cdot,F_{\widetilde{\sigma}}(\cdot))$. 
Then by Proposition~\ref{univ_deriv} we obtain  
\begin{equation}\label{deriv_prop316_proof_eq2}
\widetilde{\nabla}\widetilde{\sigma}(T(\zeta_\rho)X)
=(\rho_A(u)\Hc_B,p_{\rho_A(u)\Hc_B}(\de F_{\widetilde{\sigma}}(T(\zeta_\rho)X))).
\end{equation}

Let $\Rc$ denote  the second component of $\delta_\rho$ (as in \cite{BG11}).
By using the equality $\delta_\rho\circ\sigma=\widetilde{\sigma}\circ\zeta_\rho$ 
again, we get 
$F_{\widetilde{\sigma}}\circ\zeta_\rho=\Rc\circ\sigma\colon Z\to\Hc_A$, 
hence by differentiation we obtain 
\begin{equation}\label{deriv_prop316_proof_eq3}
\de F_{\widetilde{\sigma}}\circ T(\zeta_\rho)=\de(\Rc\circ\sigma).
\end{equation}

It now follows by 
\eqref{deriv_prop316_proof_eq1}--\eqref{deriv_prop316_proof_eq3} 
that, for all $u\in G_A$ and $X\in\mg$,
\begin{equation}\label{deriv_prop316_proof_eq4}
\rho_A(u)h=p_{\rho_A(u)\Hc_B}(\de(\Rc\circ\sigma)(X))\in\Hc_A. 
\end{equation}

Now pick any $[(u,X)]\in G_A\times_{G_B}\mg=T(G_A/G_B)$ and set 
$u(t):=u\exp_{G_A}(tX)$, for all $t\in\R$.
Then 
$$
\de(\Rc\circ\sigma)(X)
=\frac{\de}{\de t}\Big\vert_{t=0}\Rc(\sigma(u(t))
=\rho_A(u)(\de\rho_A(X)\phi(u)+\de\phi(u,X)),
$$
and therefore, by using \eqref{deriv_prop316_proof_eq4}, we obtain 
$$
\begin{aligned}
h&=\rho_A(u)^{-1}p_{\rho_A(u)\Hc_B}\rho_A(u)(\de\rho_A(X)\phi(u)+\de\phi(u,X))\\
&=p_{\Hc_B}(\de\rho_A(X)\phi(u)+\de\phi(u,X))\\
&
=\de\phi(u,X)+p_{\Hc_B}(\de\rho_A(X)\phi(u))
\end{aligned}
$$
since $\phi\colon G_A\to\Hc_B$. 
We also used the fact that  $p_{\rho_A(u)\Hc_B}
=\rho_A(u)p_{\Hc_B}\rho_A(u)^{-1}$ for all $u\in G_A$, since $\rho_A$ is a unitary representation.
Then, because of the way $h\in\Hc_B$ was chosen, we have 
$(\nabla\sigma)([(u,X)])
=[(u,\de\phi(u,X)+p_{\Hc_B}(\de\rho_A(X)\phi(u))]$, 
as asserted.
\end{proof}

\begin{remark}
\normalfont
Since the tautological bundle $\Pi_{\Hc,\Sc_0}\colon\Tc_{\Sc_0}(\Hc)\to\Gr_{\Sc_0}(\Hc)$ is diffeomorphic to its homogeneous version 
$\U (\Hc)\times_{\U(p_{\Sc_0})}\Sc_0\to\U (\Hc)/\U(p_{\Sc_0})$, the two expressions of the covariant derivative, associated with the natural connection on $\Pi_{\Hc,\Sc_0}$, given in Proposition \ref{univ_deriv} and Proposition \ref{homogcovderiv} must coincide.

In fact, using the notations of those propositions and identifying $u\Sc_0=u\U(p_{\Sc_0})$ we have that 
$F_\sigma(u\Sc_0)\equiv u\phi(u\Sc_0)$, where we are considering 
$\phi\colon\U (\Hc)/\U(p_{\Sc_0})\to\Sc_0$ rather than its $\U(p_{\Sc_0})$-equivariant version from $\U (\Hc)$ into $\Sc_0$. Then, for 
$X=uY\equiv[(u,Y)]\in T_{u\Sc_0}(\Gr_{\Sc_0}(\Hc))$ with $Y\in \Ker E_{p_{\Sc_0}}$, by differentiating in 
$F_\sigma(u\Sc_0)=u\phi(u\Sc_0)$ we obtain
$\de F_\sigma(X)=u(Y\tilde\phi(u)+\de\phi(u,Y))$, hence 
$$
\begin{aligned}
P_{u\Sc_0}(\de F_\sigma(X))&=P_{u\Sc_0}(u(Y\phi(u)+\de\phi(u,Y)))\\
&=uP_{\Sc_0}(Y\phi(u)+\de\phi(u,Y))\\
&\equiv[(u,P_{\Sc_0}(Y\phi(u))+\de\phi(u,Y))]
\end{aligned}
$$
as it was claimed.
\end{remark}

\subsection{Differential geometric aspects of completely positive mappings}\label{subsect_cp}

In this final part of the paper we will discuss 
some geometric interpretations of the completely positive mappings. 
More specifically, we will take a fresh look at the Stinespring dilations of completely positive maps 
from the perspective of the reproducing kernels and the corresponding covariant derivatives, 
as set forth in the preceding sections.
To this end, let $\Psi\colon A\to\Bc(\Hc_0)$ be a completely positive map with the Stinespring dilation $\lambda\colon A\to\Bc(\Hc)$ given by the equation
\begin{equation}\label{dilation}
\Psi(a)=V^*\lambda(a)V \quad (a\in A),
\end{equation}
and satisfying the minimality condition $\Hc=\overline{\spann}(\lambda(A)\Hc_0)$, 
where $V\colon\Hc_0\to\Hc$ is an isometry.

First of all, it is clear that setting $K^\Psi(s,t):=\Psi(s^{-1}t)$ and 
$K^\lambda(s,t):=\lambda(s^{-1}t)$, for all $s,t\in\U_A$, we get $K^\Psi$ and $K^\lambda$ two reproducing kernels for the trivial vector bundles $\U_A\times\Hc_0\to\U_A$ and 
$\U_A\times\Hc\to\U_A$, respectively. 
Moreover, the mapping 
$\Theta_V=(\id_{\U_A}\times V,\id_{\U_A})$ is a morphism between the two preceding vector bundles, for which the equality 
$\Psi(a)=V^*\lambda(a)V$ ($a\in A$) is equivalent to the fact that $K^\Psi$ 
is the pullback kernel of $K^\lambda$ through $\Theta_V$. 
That is,
$$  
\Theta_V^*K^\lambda=K^\Psi.
$$
As regards classifying morphisms, first recall that 
$V^*V=\id_\Hc$, $VV^*=P_{V(\Hc_0)}$. 
Put $\Sc_0:=V(\Hc_0)$. 
Then the natural kernel, associated with $\lambda$, for the bundle $\U_A\times \Sc_0\to\U_A$, is
$$  
K_0^\lambda(s,t):=P_{\Sc_0}\lambda(s^{-1}t)\iota_{\Sc_0} \quad (s,t\in\U_A).
$$
Incidentally, note that the equality (\ref{dilation}) can be written alternatively as
\begin{equation}\label{compodilat}
\Psi(a)=V^*P_{\Sc_0}\lambda(a)\iota_{\Sc_0}V  \quad (a\in A),
\end{equation}
since 
$$
\Psi(a)=V^*\lambda(a)V
=(V^*V)V^*\lambda(a)V(V^*V)
=V^*P_{\Sc_0}\lambda(a)\iota_{\Sc_0}V
$$
for $a\in A$.
Let $\xi=(s,h)$ be in the fiber $\{s\}\times \Sc_0$. Then for all $t\in\U_A$,
$$
(K_0^\lambda)_\xi(t)
=K_0^\lambda(t,s)(s,h)\equiv(P_{\Sc_0}\lambda(t^{-1}s)\iota_{\Sc_0})(s,h)
=P_{\Sc_0}\lambda(t^{-1})\lambda(s)h
$$
whence it follows that $(K_0^\lambda)_\xi$ can be identified with $\lambda(s)h$ as the function acting on $t$ as above. Thus the classifying morphism $\Delta_\lambda$ for the kernel 
$K_0^\lambda$ is 
$$
\begin{CD}
(s,h)\in \U_A\times \Sc_0 @>{\delta_\lambda}>> (\lambda(s)\Sc_0,
\lambda(s)h)\in\Tc_{\Sc_0}(\Hc) \\
@V{p_{\Sc_0}}VV @VV{\Pi_{\Hc,\Sc_0}}V \\
s\in \U_A @> {\zeta_\lambda}>> \lambda(s)\Sc_0\in\Gr_{\Sc_0}(\Hc)\,.
\end{CD}
$$

As a matter of fact, on account that the transpose mapping of 
$(\delta_\lambda)_s\equiv\lambda(s)$ is equal to $\lambda(s^{-1})$ for each $s\in\U_A$, one obtains that
$\Delta_\lambda^*Q_{\Hc,\Sc_0}=K_0^\lambda$, as it had to be from the universal theorem for kernels.   

Thus the classifying morphism $\Delta_\Psi$ for $K^\Psi$ is
$$
\begin{CD}
\U_A\times\Hc_0 @>{\id_{\U_A}\times V}>> \U_A\times\Sc_0 @>{\delta_\lambda}>> \Tc_{\Sc_0}(\Hc) \\
@V{P_{\U_A}}VV @VV{P_{\U_A}}V @VV{\Pi_{\Hc,\Sc_0}}V \\
\U_A @>{\id_{\U_A}}>> \U_A @>{\zeta_\lambda}>> \Gr_{\Sc_0}(\Hc)\,,
\end{CD}
$$
and the universal theorem tells us that, for $s,t\in\U_A$,
$$ 
\begin{aligned}
\Psi(s^{-1}t)=K^\Psi(s,t)=\Delta_\Psi^*Q_{\Hc,\Sc_0}(s,t)
&
=(\Delta_\lambda\Theta_V)^*Q_{\Hc,\Sc_0}(s,t)\\
=\Theta_V^*\Delta_\lambda^*Q_{\Hc,\Sc_0}(s,t)
&
=\Theta_V^*K_0^\lambda(s,t)
=V^*P_{\Sc_0}\lambda(s^{-1}t)\iota_{\Sc_0}V.
\end{aligned}
$$
On the other  hand $A={\rm span}\,\U_A$, hence the latter equality is equivalent to \eqref{compodilat}. 
In other words, the Stinespring dilation theorem, summarized in the formula \eqref{dilation}, 
can be regarded as an instance of the universality theorem for reproducing kernels of vector bundles, 
in the sense of \cite{BG11}. 

We have shown that a completely positive map $\Psi$ can be viewed as a reproducing kernel, under the form 
$(s,t)\mapsto\Psi(s^{-1}t)$. 
Let us compute the connection and covariant derivative associated with the above interpretation.

For $f\in\Hc_0$, $a\in\Ug_A$ and $\Sc_0=V(\Hc_0)$,
$$
E_{p_{\Sc_0}}(\de\lambda(a))Vf=E_{p_{\Sc_0}}(\lambda(a))Vf
=p_{\Sc_0}(\lambda(a)Vf)=VV^*\lambda(a)Vf=V\Psi(a)f.
$$

Then, using the classifying quasimorphism $\Delta_\Psi$ and the corresponding 
pull-back operation for connections, we have that the natural connection on the bundle    
$\U_A\times\Hc_0\to\U_A$ for the kernel $\Psi$ is obtained as the composition
$$ 
\begin{aligned}
((s,a),(f,h))&\mapsto ((\lambda(s),\de\lambda(a)),(Vf,Vh))\\
&\mapsto((\lambda(s),0),(Vf,Vh+E_{p_{\Sc_0}}(\de\lambda(a)Vf))\\
&=((\lambda(s),0),(Vf,Vh+V\Psi(a)f))
\mapsto((s,0),(f,h+\Psi(a)f)).
\end{aligned}
$$
for every $s\in\U_A$, $a\in\Ug_A$, $f,h\in\Hc_0$. 
In other words, the completely positive map $\Psi$ can be regarded as a connection $\Phi_\Psi$ on the trivial bundle in the form of the correspondence $\Phi_\Psi\colon (f,h)\mapsto h+\Psi(a)f$.

To compute the covariant derivative $\nabla_\Psi$ of the connection $\Phi_\Psi$, 
note that there exists a surjective isometry 
$\iota_\lambda\colon\Hc\to\Hc^{K_0^\lambda}$ defined by $\iota_\lambda(h)=P_{\Sc_0}\lambda(\,\cdot\,)^{-1}h$. 
Then, using an argument similar to that one of Propositions \ref{homoconinv}~and~\ref{homogcovderiv}, 
we find that the covariant derivative associated to the kernel $K_0^\lambda$ is given by
$$
\nabla_\lambda\sigma_0(u,a)=\de\sigma_0(u,a)
+P_{\Sc_0}(\lambda(a)\sigma_0(u))
$$
for all $u\in\U_A$, $a\in\Ug_A$ and every section $\sigma_0\colon\U_A\to\Sc_0$ of the bundle 
$\U_A\times\Sc_0\to\U_A$.

Take now any section $\sigma\colon\U_A\to\Hc_0$ of the bundle $\U_A\times\Hc_0\to\U_A$ and put 
$\sigma_0:=V\sigma\colon\U_A\to\Sc_0$. Since $\de\sigma_0=V\de\sigma$ and $P_{\Sc_0}=VV^*$, 
by using Corollary \ref{deriv_cor} we obtain for $u\in\U_A$ and $a\in\Ug_A$, that
$$
\nabla_\Psi\sigma((u,a))=\de\sigma(u,a)+\Psi(a)\sigma(u). 
$$
The completely positive map $\Psi$ can thus be interpreted in terms of covariant derivatives.

\appendix

\section{On linear connections and their pull-backs}\label{Sect2}

For the reader's convenience, we record here some general facts on connections on 
Banach fiber bundles that are needed in the present paper. 
We use \cite{KM97a} and \cite{La01} as the main references, 
but we will also provide proofs for some results where 
we were unable to find convenient references in the literature.

\subsection{Connections on fiber bundles}

\begin{definition}\label{def11}
\normalfont
Let $\varphi\colon M\to Z$ a fiber bundle and consider both vector bundle structures of the tangent space $TM$:
\begin{itemize}
\item[$\bullet$] $\tau_M\colon TM\to M$, the tangent bundle of the total space $M$.
\item[$\bullet$] $T\varphi\colon TM\to TZ$, the tangent map of $\varphi$.
\end{itemize}
A {\it connection} on the bundle $\varphi\colon M\to Z$ is a smooth map
$\Phi\colon TM\to TM$ with the following properties:
\begin{itemize}
\item[{\rm(i)}] $\Phi\circ\Phi=\Phi$;
\item[{\rm(ii)}] the pair $(\Phi,\id_M)$ is an endomorphism of the bundle
$\tau_M\colon TM\to M$;
\item[{\rm(iii)}] for every $x\in M$, if we denote $\Phi_x:=\Phi|_{T_xM}\colon
T_xM\to T_xM$, then we have $\Ran(\Phi_x)=\Ker(T_x\varphi)$,
so that we get an exact sequence
$$
0\rightarrow H_xM\hookrightarrow T_xM\mathop{\longrightarrow}\limits^{\Phi_x}
T_xM
\mathop{\longrightarrow}\limits^{T_x\varphi} T_{\varphi(x)}Z
\rightarrow 0.
$$
\end{itemize}
Here $H_xM:=\Ker(\Phi_x)$ is a closed linear subspace of $T_xM$ called the space of {\it horizontal vectors} at
$x\in M$. Similarly, the space of {\it vertical vectors} at $x\in M$ is $\Vc_xM:=\Ker(T_x\varphi)$. Then we have the direct sum decomposition $T_xM=H_xM\oplus\Vc_xM$, for every $x\in M$ (cf. \cite[subsect. 37.2]{KM97a}).
\end{definition}

We consider in this paper two special types of connections.
\begin{enumerate}
\item If $\varphi\colon M\to Z$ is a principal bundle with structure group $G$ acting to the right on $M$ by
$$
(x,g)\mapsto \mu_g(x)=\mu(x,g),\ M\times G\to M
$$
then a connection $\Phi$ on $\varphi\colon M\to Z$ is called {\it principal} whenever it is $G$-equivariant, that is,
$$
T( \mu_g)\circ\Phi=\Phi\circ T( \mu_g)
$$
for all $g\in G$ (cf. \cite[subsect. 37.19]{KM97a}).

\item
If $\varphi\colon M\to Z$ is a vector bundle then a connection $\Phi$ on $\varphi\colon M\to Z$ is called {\it linear} if the pair $(\Phi, \id_{TZ})$ is an endomorphism of the vector bundle
$T\varphi\colon TM\to TZ$ (i.e., if $\Phi$ is linear on the fibers of the bundle $T\varphi$); see \cite[subsect. 37.27]{KM97a}.
\end{enumerate}
We are interested in particular in vector bundles constructed out of principal ones. 
Recall how they appear:
Let $\pi\colon\Pc\to Z$ be a principal Banach bundle
with the structure Banach-Lie group $G$ and
the action $\mu\colon \Pc\times G\to \Pc$.
Assume that $\rho\colon G\to\Bc(\textbf{E})$ is
a smooth representation of $G$ by linear operators on
a Banach space $\textbf{E}$, and denote by
$$
[(p,e)]\mapsto\pi(p),\quad\Pi\colon D=\Pc\times_G\textbf{E}\to Z 
$$
the \textit{associated vector bundle}
(see \cite[subsect. 6.5]{Bo67} and \cite[subsect. 37.12]{KM97a}).
Here $\Pc\times_G\textbf{E}$ denotes the quotient of
$\Pc\times\textbf{E}$ with respect to the equivalence relation defined by
$$
(\forall g\in G)\quad (p,e)\sim(\mu(p,g),\rho(g^{-1})e)=:\bar{\mu}(g)(p,e)
$$
whenever $(p,e)\in \Pc\times\textbf{E}$,
and we denote by $[(p,e)]$ the equivalence class of any pair $(p,e)$.

In this way, $\Pi\colon\Pc\times_G\textbf{E}\to Z$ is a vector $G$-bundle.

\begin{remark}\label{conninduc}
\normalfont
Every connection on a principal bundle $\pi$ induces a linear connection on any vector bundle associated to $\pi$.   
A good reference for that induction procedure in infinite dimensions is \cite{KM97a}. 
We will recall here the corresponding construction 
since we will need it in order to describe specific induced connections 
later on (see for instance the comment prior to Theorem~\ref{conred} below).

For a Banach-Lie group $G$ with the Lie algebra $\gg=T_{\1}G$ 
let $\lambda_g\colon G\to G$, $\lambda_g(h)=gh$ for all $g,h\in G$. 
Then the mapping $(g,X)\mapsto T_{\1}(\lambda_g)X$ 
is a diffeomorphism $G\times\gg\to TG$, 
and thus 
the tangent manifold $TG$ is endowed with structure of a semidirect product of groups 
$TG\equiv G \ltimes_{\Ad_{G}}\Gg$ defined by the adjoint action of $G$ on $\Gg$; 
see \cite[Cor. 38.10]{KM97a}. 
The
multiplication in the group $TG$ is given by 
$$
(g_1,X_1)(g_2,X_2)=(g_1 g_2,\Ad_{G}(g_2^{-1})X_1+X_2),
\quad (g_1,g_2\in G; X_1,X_2\in\Gg).
$$
Let $\pi\colon \Pc\to Z$ be a principal bundle with the structure group $G$ acting to the right by  
$\mu\colon \Pc\times G\to \Pc$. 
If $\rho\colon G\to \Bc(\bf E)$ is a smooth representation as above, 
then we can form the associated vector bundle $\Pi\colon D=\Pc\times_{G}\textbf{E}\to Z$. 

For describing a connection induced on $\Pi$, 
one needs a specific description of the tangent space 
of the total space $\Pc\times_{G}\textbf{E}$, and to this end one uses the fact that 
 the tangent functor commutes with the construction of associated bundles. 
In fact, the tangent bundle $T\pi\colon T\Pc\to TZ$ is a principal bundle with the structure 
group $TG=G\ltimes_{\Ad_{G}}\Gg$ and right action $T\mu\colon T\Pc\times TG\to T\Pc$  
(\cite[Th. 37.18(1)]{KM97a}). 
The representation $\rho$ gives a linear action $G\times\textbf{E}\to\textbf{E}$, 
and by computing the tangent map of that action it follows that 
the tangent map of the above representation can be viewed as the smooth representation
$T\rho\colon G\ltimes_{\Ad_{G}}\Gg\to\Bc(\textbf{E}\oplus\textbf{E})$, 
which is easily computed as 
$$(g,X)\mapsto 
\begin{pmatrix}
\rho(g)&0\cr 0 & \rho(g)\end{pmatrix}
\begin{pmatrix} 1& 0 \cr d\rho(X)  & 1\end{pmatrix}
=\begin{pmatrix} \rho(g)& 0 \cr \rho(g) d\rho(X) & \rho(g)
\end{pmatrix},
$$
where the resulting matrix is to be understood as acting on vectors of $\textbf{E}\oplus\textbf{E}$ written in column form.
Using the representation $T\rho$, the tangent bundle  
of the vector bundle $\Pi\colon D=\Pc\times_{G}\textbf{E}\to Z$  
can be described as the vector bundle 
$$\tau_D\colon TD=T\Pc\times_{G\ltimes_{\Ad_{G}}\Gg}(\textbf{E}\oplus\textbf{E})
\to \Pc\times_{G}\textbf{E}=D,$$
which is associated to the principal bundle $T\pi\colon T\Pc\to TZ$ and is defined by 
$$
\tau_D\colon [(v_p,(f,h))]\mapsto [(p,f)] \qquad (v_p\in T_p\Pc; f,h\in\bf E)
$$ 
(\cite[Th. 37.18(4)]{KM97a}).

If now $\Phi\colon T\Pc\to T\Pc$ is a principal connection on the principal bundle 
$\pi\colon \Pc\to Z$,   
then the mapping
$\Phi\times\id_{T\textbf{E}}\colon T\Pc\times T\textbf{E}\to T\Pc\times T\textbf{E}$ is
$TG$-equivariant and 
factorizes through a map
$\bar\Phi\colon T(\Pc\times_G\textbf{E})\to T(\Pc\times_G\textbf{E})
=T\Pc\times_{TG}T\textbf{E}$. 
That is, there exists the commutative diagram 
$$
\begin{CD}
T\Pc\times T\textbf{E} @>{\Phi\times\id_{T\textbf{E}}}>> T\Pc\times T\textbf{E} \\
@V{ }VV @VV{ }V \\
T\Pc\times_{TG}T\textbf{E} @>{\bar\Phi}>> T\Pc\times_{TG}T\textbf{E}
\end{CD}
$$
and $\bar\Phi$ is the connection \textit{induced}  by $\Phi$ on $T(\Pc\times_G\textbf{E})$;   
see \cite[subsect. 37.24]{KM97a}.
\end{remark}

We now briefly recall the covariant derivatives (as in \cite{Vi67}) and we then provide a proposition 
needed in the specific computations carried out in the present paper.

Let $\Pi\colon D\to Z$ be a vector bundle with a linear connection $\Phi\colon TD\to TD$. 
Let $\Vc D=\Ker(T\Pi)$ ($\subseteq TD$) be the vertical part of the tangent bundle $\tau_D\colon TD\to D$. 
A useful description of $\Vc D$ can be obtained by considering the fibered product 
$D\fimes_Z D:=\{(x_1,x_2)\in D\times D\mid\Pi(x_1)=\Pi(x_2)\}$ 
along with the natural maps $r_j\colon D\fimes_Z D\to D$, $r_j(x_1,x_2)=x_j$ for $j=1,2$. 
Define for every $(x_1,x_2)\in D\fimes_Z D$ the path 
$c_{x_1,x_2}\colon\R\to D$, $c_{x_1,x_2}(t)=x_1+tx_2$. 
Then it is easily seen that we have a well-defined diffeomorphism 
$\varepsilon\colon D\fimes_Z D\to \Vc D,\quad \varepsilon(x_1,x_2)=\dot c_{x_1,x_2}(0)\in T_{x_1} D$, 
which is in fact an isomorphism between the vector bundles 
$r_1\colon D\fimes_Z D\to D$ and $\tau_D\vert_{\Vc D}\colon \Vc D\to D$. 
We then get a natural mapping 
$r:=r_2\circ\varepsilon^{-1} \colon\Vc D\to D$ and the pair 
$(r,\Pi)$ is a homomorphism of vector bundles from $\tau_D\vert_{\Vc D}\colon \Vc D\to D$ 
to $\Pi\colon D\to Z$. 

Next let $\Omega^1(Z,D)$ the space of locally defined 
smooth differential 1-forms on $Z$ with values in the bundle $\Pi\colon D\to Z$,  
hence the set of smooth mappings $\eta\colon\tau_Z^{-1}(Z_\eta)\to D$, 
where $\tau_Z\colon TZ\to Z$ is the tangent bundle and  $Z_\eta$ is a suitable open subset of $Z$,  
such that for every $z\in Z_\eta$ we have a bounded linear operator 
$\eta_z:=\eta\vert_{T_zZ}\colon T_zZ\to D_z=\Pi^{-1}(z)$. 
(So the pair $(\eta,\id_Z)$ is a homomorphism of vector bundles 
from the tangent bundle $\tau_D\vert_{Z_\eta}$ to the bundle~$\Pi$.)
For the sake of simplicity we actually omit the subscript $\eta$ 
in $Z_\eta$, as if the forms were always defined throughout $Z$; 
in fact, the algebraic operations are performed on the intersections of the domains, 
and so on.  
Similarly, we let $\Omega^0(Z,D)$ be the space of locally defined smooth sections of the vector bundle~$\Pi$. 

\begin{definition}\label{deriv_def}
\normalfont
The \emph{covariant derivative} for the linear connection $\Phi$ 
is the linear mapping  
$\nabla\colon\Omega^0(Z,D)\to\Omega^1(Z,D)$, defined for every $\sigma\in\Omega^0(Z,D)$ 
by the composition 
$$\nabla\sigma\colon  
TZ\mathop{\longrightarrow}\limits^{T\sigma} TD\mathop{\longrightarrow}\limits^{\Phi} 
\Vc D\mathop{\longrightarrow}\limits^{r}D$$
that is, $\nabla\sigma=(r\circ\Phi)\circ T\sigma$. 
(The composition $r\circ\Phi$ is the so-called \emph{connection map}.)
\end{definition}

\begin{proposition}\label{deriv_prop1}
Let $\Pi\colon D\to Z$ and $\widetilde{\Pi}\colon\widetilde{D}\to\widetilde{Z}$ be vector bundles 
endowed with the linear connections $\Phi$ and $\widetilde{\Phi}$, 
with the corresponding covariant derivatives $\nabla$ and $\widetilde{\nabla}$, 
respectively. 
Assume that $\Theta=(\delta,\zeta)$ is a homomorphism of vector bundles from $\Pi$ into $\widetilde{\Pi}$ 
such that $T\delta\circ\Phi=\widetilde{\Phi}\circ T\delta$. 
If $\sigma\in\Omega^0(Z,D)$ and $\widetilde{\sigma}\in\Omega^0(\widetilde{Z},\widetilde{D})$ are 
such that $\delta\circ\sigma=\widetilde{\sigma}\circ\zeta$, 
then  $\delta\circ\nabla\sigma=\widetilde{\nabla}\widetilde{\sigma}\circ T\zeta$. 
\end{proposition}

\begin{proof}
First let $r\colon \Vc D\to D$ and $\widetilde{r}\colon\Vc\widetilde{D}\to\widetilde{D}$ be the natural mappings 
and note that 
\begin{equation}\label{deriv_prop1_proof_eq1}
\delta\circ r=\widetilde{r}\circ T\delta.
\end{equation} 
In order to see why this equality holds true we need the mapping 
$\delta\fimes_Z\delta\colon D\fimes_Z D\to\widetilde{D}\fimes_{\widetilde{Z}}\widetilde{D}$ 
given by $(\delta\fimes_Z\delta)(x_1,x_2)=(\delta(x_1),\delta(x_2))$, 
which is well defined since $\widetilde{\Pi}\circ\delta=\zeta\circ\Pi$.  
Since $\delta$ is fiberwise linear, it follows that with the notation of Definition~\ref{deriv_def} 
we have $\delta\circ c_{x_1,x_2}=c_{\delta(x_1),\delta(x_2)}\colon\R\to\widetilde{D}$ 
for all $(x_1,x_2)\in D\fimes_Z D$. 
By taking the velocity vectors at $0\in\R$ for these paths we get 
$T\delta\circ\varepsilon=\widetilde{\varepsilon}\circ(\delta\fimes_Z\delta)\colon D\fimes_Z D\to T\widetilde{D}$. 
Therefore 
$\widetilde{\varepsilon}^{-1}\circ T\delta=(\delta\fimes_Z\delta)\circ\varepsilon^{-1}$ 
and then, by using the obvious equality 
$\widetilde{r}_2\circ(\delta\fimes_Z\delta)=\delta\circ r_2\colon D\fimes_Z D\to\widetilde{D}$, 
we get 
$$\widetilde{r}\circ T\delta
=\widetilde{r}_2\circ \widetilde{\varepsilon}^{-1}\circ T\delta
=\widetilde{r}_2\circ(\delta\fimes_Z\delta)\circ\varepsilon^{-1}
=\delta\circ r_2\circ\varepsilon^{-1}
=\delta\circ r $$
hence \eqref{deriv_prop1_proof_eq1} holds true. 

We now come back to the proof of the assertion. 
By using \eqref{deriv_prop1_proof_eq1} 
and the equality $T\delta\circ\Phi=\widetilde{\Phi}\circ T\delta$ 
we get 
\begin{equation}\label{deriv_prop1_proof_eq2}
\delta\circ(r\circ\Phi)=\widetilde{r}\circ T\delta\circ\Phi =(\widetilde{r}\circ\widetilde{\Phi})\circ T\delta. 
\end{equation}
On the other hand we have $\delta\circ\sigma=\widetilde{\sigma}\circ\zeta$, 
and therefore $T\delta\circ T\sigma=T\widetilde{\sigma}\circ T\zeta$. 
We then get 
$$\widetilde{\nabla}\widetilde{\sigma}\circ T\zeta
=(\widetilde{r}\circ\widetilde{\Phi})\circ T\widetilde{\sigma}\circ T\zeta
=(\widetilde{r}\circ\widetilde{\Phi})\circ T\delta\circ T\sigma
=\delta\circ(r\circ\Phi)\circ T\sigma
=\delta\circ\nabla\sigma
$$
where the next-to-last equality follows by \eqref{deriv_prop1_proof_eq1}, and this completes the proof. 
\end{proof}

\subsection{Pull-backs of connections}\label{sect2}

Pull-backs of connections on various types of finite-dimensional bundles
have been studied in several papers; see for instance
\cite{NR61}, \cite{NR63},
\cite{Le68},
\cite{Sch80}, 
\cite{PR86}.
We now establish a result (Proposition~\ref{prop13}) that belongs to that circle of ideas
and is appropriate for the applications we want to make
in infinite dimensions. 
Unlike the descriptions of the pull-backs of connections that we were able to find in the literature, 
the method provided here is more direct in the sense that it requires 
neither the connection map, nor any connection forms, nor the covariant derivative, 
but rather the connection itself. 
The intertwining property of the covariant derivatives follows at once 
(Corollary~\ref{deriv_cor}).

We will need the following simple lemma.

\begin{lemma}\label{lemacon}
Let $T\colon{\Ec}\to\widetilde{\Ec}$ be a continuous (conjugate-)linear operator between two Banach spaces $\Ec$ and $\widetilde{\Ec}$. 
Let us assume that there are two closed linear subspaces $\Fc\subset\Ec$ and $\widetilde{\Fc}\subset\widetilde{\Ec}$ such that:
\begin{itemize}
\item[{\rm(i)}]  the operator $T$ induces a (conjugate-)linear isomorphism $T|_{\Fc}\colon\Fc\to\widetilde{\Fc}$;
\item[{\rm(ii)}] $\Ran\widetilde{P}=\widetilde{\Fc}$, for some projection
$\widetilde P\colon\widetilde{\Ec}\to\widetilde{\Ec}$.
\end{itemize}
Then there exists a unique projection $P\in\End(\Ec)$ such that
$\Ran P=\Fc$ and ${\widetilde P}\circ T=T\circ P$.
\end{lemma}

\begin{proof}
{\it Existence:} Define
\begin{equation}\label{defproy}
P:=(T|_{\Fc})^{-1}\circ\widetilde P\circ T\in\End(\Ec)
\end{equation}
It is clear that $\Ran P=\Fc$ and moreover $P|_{\Fc}=\hbox{id}_{\Fc}$, hence
$P\circ P=P$. Then the
commutativity of the diagram is satisfied by the construction of $P$.

{\it Uniqueness:} Assume that $P_1\in\End(\Ec)$ is another operator satisfying the properties of the statement. Then for arbitrary $x\in\Ec$ we have
$T(P_1x)=\widetilde{P} Tx=T(Px)$. 
Since $P_1x,Px\in\Fc$ and $T|_{\Fc}\colon\Fc\to\widetilde{\Fc}$ is an isomorphism, it then follows that
$P_1x=Px$. Thus $P_1=P$ and we are done.
\end{proof}

\begin{proposition}\label{prop13}
Let $\varphi\colon M\to Z$ and $\widetilde\varphi\colon \widetilde M\to \widetilde Z$ 
be fiber bundles modeled on Banach spaces, and let $\Theta=(\delta,\zeta)$ be a bundle homomorphism, that is, the diagram
$$
\begin{CD}
M @>{\delta}>> \widetilde M \\
@V{\varphi}VV @VV{\widetilde\varphi}V \\
Z @>{\zeta}>> \widetilde Z
\end{CD}
$$
is commutative and both $\delta$ and $\zeta$ are smooth.
In addition, assume that for every $s\in Z$ the mapping $\delta$ induces a diffeomorphism of the fiber
$M_s:=\varphi^{-1}(\{s\})$ onto the fiber $\widetilde{M}_{\zeta(s)}
:=\widetilde\varphi^{-1}(\zeta(s))$.

Then for every connection $\widetilde\Phi$ on the bundle $\widetilde\varphi\colon\widetilde M\to\widetilde Z$ 
there exists a unique connection $\Phi$ on the bundle
$\varphi\colon M\to Z$ such that the diagram
$$
\begin{CD}
TM @>{T\delta}>> T\widetilde M \\
@V{\Phi}VV @VV{\widetilde\Phi}V \\
TM @>{T\delta}>> T\widetilde M
\end{CD}
$$
is commutative.

Moreover, if both $\varphi\colon M\to Z$ and $\widetilde\varphi\colon\widetilde M\to\widetilde Z$ are principal (vector) bundles, 
the pair $\Theta=(\delta,\zeta)$ is a homomorphism of principal bundles (or of vector bundles, and in this case $\delta$ can be linear) bundles, and $\widetilde\Phi$ is a principal (linear or conjugate-linear) connection, then so is $\Phi$.
\end{proposition}

\begin{proof}
We have for every $x\in M$ the continuous operator 
$T_x\delta\colon T_xM\to T_{\delta(x)}\widetilde M$ 
(which is either linear or conjugate-linear), and also the relations
$T_x(M_{\varphi(x)}) = \Vc_x\hookrightarrow T_xM$ and
$$
T_{\delta(x)}(\widetilde M_{\zeta(\varphi(x))})=T_{\delta(x)}(\widetilde M_{\widetilde\varphi(\delta(x))})
=\Vc_{\delta(x)}\widetilde M\hookrightarrow T_{\delta(x)}\widetilde M.
$$
Since $\delta |_{M_{\Pi(x)}}\colon M_{\Pi(x)}\to \widetilde M_{\zeta(\Pi(x))}$ is a diffeomorphism by hypothesis, 
it thus follows that the operator $T_x\delta$ induces a (conjugate-)linear isomorphism
$\Vc_xM\to \Vc_{\delta(x)}\widetilde M$. 
Now Lemma~\ref{lemacon} shows that 
there exists a unique idempotent operator
$\Phi_x\colon T_xM\to T_xM$ such that $\Ran\Phi_x=\Vc_xM$ and
$(T_x\delta)\circ\Phi_x=\widetilde\Phi_{\delta(x)}\circ(T_x\delta)$. 
In fact it is defined by
$$
\Phi_x:=(T_x\delta |_{\Vc_x M})^{-1}\circ\widetilde\Phi_{\delta(x)}\circ T_x\delta \qquad (x\in M).
$$
If we put together the operators $\Phi_x$ with $x\in M$, we get the map $\Phi\colon TM\to TM$ we were looking for. 
What still remains to be done is to check that $\Phi$ is smooth. 
Since this is a local property, we may assume that both bundles $\Pi$ and $\widetilde\Pi$ are trivial.
 Let $S$ and $\widetilde S$ be their typical fibers, respectively. 
 Then $M=Z\times S$ and $\widetilde M=\widetilde Z\times\widetilde S$, hence $TM=TM\times TS$ and
$T\widetilde M=T\widetilde Z\times T\widetilde S$. 
The fact that $\widetilde\Phi$ is a connection means that for every
$(\widetilde z,\widetilde k)\in\widetilde Z\times\widetilde S$ we have an idempotent operator
$\widetilde\Phi_{(\widetilde z, \widetilde k)}$ on
$T_{\widetilde z}\widetilde Z\times T_{\widetilde k}\widetilde S$ with
$\Ran\widetilde\Phi_{(\widetilde z, \widetilde k)}=\{0\}\times T_{\widetilde k}\widetilde S$.

Moreover, we have the smooth map $\delta\colon Z\times S\to \widetilde Z\times\widetilde S$ 
for which there exists a smooth map $d\colon Z\times S\to\widetilde S$ such that
$\delta(z, k)=(\zeta(z),d(z,k))$ for all $z\in Z$ and $k\in S$.
The hypothesis that $\delta$ is a fiberwise diffeomorphism is equivalent to the fact that 
for every $z\in Z$ we have the diffeomorphism $d(z,\ \cdot\ )\colon S\to \widetilde S$. 
It follows by (\ref{defproy}) that, for arbitrary $(z,k)\in Z\times S$,
$$
\Phi_{(z,k)}=T_k(d(z,\ \cdot\ ))^{-1}
\circ
\widetilde\Phi_{\delta(z,k)}
\circ T_{(z,k)}\delta
\in\End(T_z Z\times T_k S)
$$
which clearly shows that $\Phi\colon TZ\times TS\to TZ\times TS$ is smooth. 
(Note that the smoothness of the mapping $(z,k)\mapsto T_k(d(z,\ \cdot\ ))^{-1}$ 
is ensured by the fact that we are working with Banach manifolds.)

The remainder of the proof is straightforward.
\end{proof}

\begin{definition}\label{pullcon}
\normalfont
In the setting of Proposition \ref{prop13} we say that the connection $\Phi$ is the 
{\it pull-back of the connection} $\widetilde\Phi$ and we denote $\Phi=\Theta^*(\widetilde\Phi)$.
\end{definition}

\begin{corollary}\label{deriv_cor}
Let $\Pi\colon D\to Z$ and $\widetilde{\Pi}\colon\widetilde{D}\to\widetilde{Z}$ be vector bundles. 
Assume that $\Theta=(\delta,\zeta)$ is a homomorphism of vector bundles from $\Pi$ into $\widetilde{\Pi}$ 
such that  for every $s\in Z$ the mapping $\delta$ induces an isomorphism of the fiber
$D_s:=\Pi^{-1}(\{s\})$ onto the fiber 
$\widetilde{D}_{\zeta(s)}:=\widetilde\Pi^{-1}(\zeta(s))$. 
Consider any linear connection $\widetilde\Phi$ on the vector bundle $\Pi$ and 
its pull-back $\Phi=\Theta^*(\widetilde\Phi)$ on the vector bundle~$\widetilde{\Pi}$, 
with the corresponding covariant derivatives $\nabla$ and $\widetilde{\nabla}$, 
respectively. 
If we have $\sigma\in\Omega^0(Z,D)$ and $\widetilde{\sigma}\in\Omega^0(\widetilde{Z},\widetilde{D})$ 
such that $\delta\circ\sigma=\widetilde{\sigma}\circ\zeta$, 
then  $\delta\circ\nabla\sigma=\widetilde{\nabla}\widetilde{\sigma}\circ T\zeta$. 
\end{corollary}

\begin{proof}
Use Propositions  \ref{prop13} and \ref{deriv_prop1}. 
\end{proof}

\subsection*{Acknowledgment} 
We wish to thank Professor Joachim Hilgert for kindly sending over one of his papers upon our request, 
and Professor Radu Pantilie for pointing out useful references and facts on linear connecctions.




\begin{thebibliography}{nunitary}



\bibitem[ACS95]{ACS95}
E.~Andruchow, G.~Corach, D.~Stojanoff, 
A geometric characterization of nuclearity and injectivity. 
\textit{J. Funct. Anal.} \textbf{133} (1995), no.~2, 474--494.


\bibitem[BG08]{BG08}
D.~Belti\c t\u a, J.E.~Gal\'e,
Holomorphic geometric models for representations of $C^*$-algebras.
\textit{J. Funct. Anal.} \textbf{255} (2008), no.~10, 2888--2932.

\bibitem[BG09]{BG09}
D.~Belti\c t\u a, J.E.~Gal\'e,
On complex infinite-dimensional Grassmann manifolds.
\textit{Complex Anal. Oper. Theory} \textbf{3} (2009), no.~4, 739-758.

\bibitem[BG11]{BG11}
D.~Belti\c t\u a, J.E.~Gal\'e,
Universal objects in categories of reproducing kernels.
\textit{Rev. Mat. Iberoamericana} \textbf{27} (2011), no.~1, 123-179.

\bibitem[BR07]{BR07}
D.~Belti\c t\u a, T.S.~Ratiu,
Geometric representation theory for unitary groups of operator algebras.
\textit{Adv. Math.} \textbf{208} (2007), no. 1, 299-317.

\bibitem[BH98]{BH98}
W.~Bertram, J.~Hilgert,
Reproducing kernels on vector bundles.  
In:
{\it Lie Theory and Its Applications in Physics III},
World Scientific, Singapore, 1998, 43-58.

\bibitem[Bo67]{Bo67}
N.~Bourbaki,
\textit{\'El\'ements de Math\'ematique. Fasc. XXXIII.
Vari\'et\'es diff\'e\-ren\-ti\-elles et analytiques}.
   Fascicule de r\'esultats (Paragraphes 1 \`a 7).
Actualit\'es Scient. et Industr., No.~1333, Hermann, Paris, 1967.

\bibitem[BR90]{BR90}
F.E.~Burstall, J.H.~Rawnsley, 
\textit{Twistor theory for Riemannian symmetric spaces}. 
Lecture Notes in Mathematics, 1424. Springer-Verlag, Berlin, 1990.

\bibitem[CG99]{CG99}
G.~Corach, J.E.~Gal\'e,
On amenability and geometry of spaces of bounded representations. 
\textit{J. London Math. Soc.} {\bf 59} (1999),
no.~2, 311--329.

\bibitem[DG01]{DG01}
M.J.~Dupr\'e, J.F.~Glazebrook, 
 The Stiefel bundle of a Banach algebra. 
 \textit{Integral Equations Operator Theory} \textbf{41} (2001), no.~3, 264--287.


\bibitem[Hi08]{Hi08}
J.~Hilgert,
Reproducing kernels in representation theory.   
In: B.~Gilligan, G.J.~Roos (eds.), 
{\it Symmetries in Complex Analysis}, Contemporary Math., 468. American Mathematical Society, Providence, RI, 2008, pp.~1--98. 

\bibitem[Ho62]{Ho62}
K.~Hoffman, 
\textit{Banach Spaces of Analytic Functions}. 
Prentice-Hall Series in Modern Analysis Prentice-Hall, Inc., Englewood Cliffs, N. J., 1962. 


\bibitem[KM97]{KM97a}
A.~Kriegl, P.W.~Michor,
\textit{The Convenient Setting of Global Analysis}.
 Mathematical Surveys and Monographs, 53.
 American Mathematical Society, Providence, RI, 1997.


\bibitem[La01]{La01}
S.~Lang,
{\it Fundamentals of Differential Geometry} (corrected second printing).
Graduate Texts in Mathematics, 191. Springer-Verlag,
New-York,
2001.


\bibitem[Le68]{Le68}
D.~Lehmann,
Quelques propri\'et\'es des connexions induites.
\textit{Bull. Soc. Math. France Suppl. M\'em.}
\textbf{16} (1968), 7--99.

\bibitem[MS97]{MS97}
M.~Martin, N.~Salinas, 
Flag manifolds and the Cowen-Douglas theory. 
\textit{J. Operator Theory} \textbf{38} (1997), no. 2, 329--365. 

\bibitem[MR92]{MR92}
L.E.~Mata-Lorenzo, L.~Recht,
Infinite-dimensional homogeneous reductive spaces.
{\it Acta Cient. Venezolana}
{\bf 43}(1992),  no.~2, 76--90.

\bibitem[MP97]{MP97}
M.~Monastyrski, Z.~Pasternak-Winiarski,
Maps on complex manifolds into Grassmann spaces defined by
   reproducing kernels of Bergman type. 
\textit{Demonstratio Math.} \textbf{30} (1997), no.~2, 465--474.

\bibitem[NR61]{NR61}
M.S.~Narasimhan, S.~Ramanan,
Existence of universal connections.
\textit{Amer. J. Math.} \textbf{83} (1961), 563--572.

\bibitem[NR63]{NR63}
M.S.~Narasimhan, S.~Ramanan,
Existence of universal connections. II.
\textit{Amer. J. Math.} \textbf{85} (1963), 223--231.

\bibitem[Ne00]{Ne00}
K.-H.~Neeb,
\textit{Holomorphy and Convexity in Lie Theory}.
 de Gruyter Expositions in Mathematics 28,
Walter de Gruyter \& Co.,
Berlin, 2000.

\bibitem[Ne02]{Ne02}
K.-H.~Neeb, 
A Cartan-Hadamard theorem for Banach-Finsler manifolds.  
\textit{Geom. Dedicata} \textbf{95} (2002), 115--156.

\bibitem[Ne10]{Ne10}
K.-H.~Neeb, 
On differentiable vectors for representations of infinite dimensional Lie groups. 
\textit{J. Funct. Anal.} \textbf{259} (2010), no.~11, 2814--2855. 

\bibitem[Ne13]{Ne12}
K.-H.~Neeb, 
Holomorphic realization of unitary representations of Banach-Lie groups.  
In: A.~Huckleberry, I.~Penkov, G.~Zuckerman (eds.), 
\textit{Lie Groups: Structure, Actions and Representations}, 
Progress in Math., 306, Birkh\"auser, 2013. 

 \bibitem[Od88]{Od88}
A.~Odzijewicz,
On reproducing kernels and quantization of states. 
\textit{Comm. Math. Phys.} \textbf{114} (1988), no.~4, 577--597.

\bibitem[Od92]{Od92}
A.~Odzijewicz,
Coherent states and geometric quantization. 
\textit{Comm. Math. Phys.} \textbf{150} (1992), no.~2, 385--413.


\bibitem[Pa02]{Pa02}
V.~Paulsen,
\textit{Completely Bounded Maps and Operator Algebras}. 
Cambridge Studies in Advanced Mathematics, 78. Cambridge University Press, Cambridge, 2002.


\bibitem[PR86]{PR86}
H.~Porta, L.~Recht,
Classification of linear connections. 
\textit{J. Math. Anal. Appl.} \textbf{118} (1986), no.~2, 547--560.


\bibitem[Sch80]{Sch80}
R.~Schlafly,
Universal connections. 
\textit{Invent. Math.} \textbf{59} (1980), no.~1, 59--65.


\bibitem[Up85]{Up85}
H.~Upmeier,
{\it Symmetric Banach Manifolds and Jordan $C^*$-algebras}. 
North-Holland Mathematics Studies, 104.
Notas de Matem\`atica, 96. North-Holland Publishing Co.,
Amsterdam,
1985.

\bibitem[Vi67]{Vi67}
J.~Vilms, 
Connections on tangent bundles. 
\textit{J. Differential Geometry} \textbf{1} (1967), 235--243. 

\bibitem[We08]{We08} 
R.O.~Wells, Jr., 
{\it Differential analysis on complex manifolds}. Third edition. With a new appendix by Oscar Garcia-Prada. 
Graduate Texts in Mathematics, 65. Springer, New York, 2008. 

\end{thebibliography}
\end{document}